\documentclass[11pt,twoside,reqno]{amsart}

\usepackage{amsmath, amsfonts, amsthm, amssymb, amscd}
\usepackage{mathrsfs}

\usepackage{graphicx}
\usepackage{tikz}
\usepackage{subcaption} 
\usepackage{float}      

\usepackage[english]{babel}
\usepackage{enumerate}

\usepackage{xcolor}
\usepackage{hyperref}

\usepackage{geometry}
\usepackage{verbatim}

\newtheorem{theorem}{Theorem}[section]
\newtheorem{lemma}[theorem]{Lemma}

\newtheorem{definition}[theorem]{Definition}

\newtheorem{proposition}[theorem]{Proposition}
\newtheorem{corollary}[theorem]{Corollary}
\newtheorem{remark}[theorem]{Remark}

\newtheorem{question}[theorem]{Question}
\newtheorem{conjecture}{Conjecture}

\numberwithin{equation}{section}
\numberwithin{figure}{section}

\newcommand{\R}{\mathbb{R}}
\newcommand{\C}{\mathbb{C}}

\renewcommand{\div}{\mathrm{div}}

\newcommand{\mx}{\mu^\xi}

\def\intave#1{\int_{#1}\hbox{\llap{$\raise2.3pt\hbox{\vrule
				height.9pt width7pt}\phantom{\scriptstyle{#1}}\mkern-2mu$}}}

\everymath{\displaystyle}

\makeatletter
\@namedef{subjclassname@2020}{Mathematics Subject Classification 2020}
\makeatother

\title{A FLOW APPROACH TO THE MONOTONICITY OF SHAPE FUNCTIONALS}
\author{Yong Huang}
\address{School of Mathematics, Hunan University, Changsha, Hunan, China.}
\email{huangyong@hnu.edu.cn}

\author{Qinfeng Li}
\address{School of Mathematics, Hunan University, Changsha, P.R. China.}
\email{liqinfeng1989@gmail.com}

\author{Shuangquan Xie}
\address{School of Mathematics, Hunan University, Changsha, P.R. China.}
\email{xieshuangquan@hnu.edu.cn}

\author{Hang Yang}
\address{School of Mathematics, Hunan University, Changsha, P.R. China.}
\email{hangyang0925@gmail.com}

\thanks{Research of Yong Huang is supported by the National Natural Science Foundation of China (No. 12171144 and No. 12231006).Research of Qinfeng Li was supported by National Natural Science Fund of China for Excellent Young Scholars (No. 12522109), National Key R\&D Program of China (No. 2022YFA1006900) and the National Science Fund of China General Program (No. 12471105). Research of Shuangquan Xie is supported by the National Science fund of China for Youth Scholars (No.12401198) and the Changsha Natural Science Foundation (No. KQ2208006).}

\begin{document}
	\maketitle	
\begin{abstract}
		We develop a geometric flow framework to investigate two classical shape functionals: the torsional rigidity and the first Dirichlet eigenvalue of the Laplacian. First, by constructing novel deformation paths governed by height-stretching flows, leg-stretching flows, and angle-bisector flows, we prove new monotonicity properties for these functionals under deformations of triangles and rhombuses. These results also lead to new and simpler proofs of some known results, without using the Steiner symmetrization argument. Second, we introduce a mean curvature flow approach to the Saint-Venant inequality, providing a new geometric proof for smooth convex domains. We establish a weak monotonicity property along the flow and characterize the equality case, which leads to the discovery of an intriguing new functional whose extremal properties suggest a further conjecture. Third, by discovering a gradient norm inequality for the sides of rectangles, we prove monotonicity and rigidity results of the torsional rigidity on rectangles. 
\end{abstract}
	
	\section{Introduction}
	\subsection*{Motivation and Background}
	A fundamental question in shape optimization concerns the identification of maximizers and minimizers of shape functionals—when such extrema exist. From a practical standpoint, however, a theoretically optimal shape may not always be feasible or appropriate for implementation. Often, it is necessary to select the best solution from a prescribed set of admissible shapes, typically choosing one that approximates the optimal shape as closely as possible. This motivates the following two questions, which we consider to be of substantial importance, both theoretically and in application.
	\begin{question}
		\label{kq1}
		Suppose the maximizer (or minimizer) of a given shape functional under appropriate constraints is unique. In what sense is it true that the closer a shape is to this maximizer (or minimizer), the larger (or smaller) the value the shape functional attains on that shape?
	\end{question}
	
	\begin{question}
		\label{kq2}
		In the absence of optimality and direct computability, how can one effectively compare the values attained by a shape functional on two generic regions?
	\end{question}
	
	
	One natural approach to address the two questions is via the continuous deformation method. More specifically, if there is a continuous deformation that transforms a given shape into the optimal one, and during this evolution the shape becomes strictly closer to the optimal one, then Question \ref{kq1} becomes the problem of studying the monotonicity of the shape functional along the continuous deformation. For Question \ref{kq2}, if neither shape is optimal, the idea is to find a continuous deformation that transforms one shape into the other, while ensuring that the shape functional is monotone along this continuous deformation. Once such a continuous deformation is found, Question \ref{kq2} is then answered. 
	
	Therefore, both questions above require us to study the monotonicity property of shape functionals under suitable continuous deformations. However, to the best of our knowledge, even for very classical functionals such as the torsional rigidity and the first eigenvalue of the Dirichlet Laplacian, the monotonicity has been less explored along other continuous deformations except through the application of continuous symmetrization arguments.
	
	The continuous symmetrization method, introduced by P\'olya and Szeg\"o in \cite{PS51} and later developed into various versions by Solynin \cite{Solynin92} and Brock \cite{Brock95, Brock00}, has demonstrated remarkable effectiveness in addressing shape optimization problems, particularly in the context of planar domains, as exemplified by Laugesen and Siudeja in \cite[Chapter 6]{nshap}. However, concerning Question \ref{kq1}, there are in general many continuous deformation paths that can gradually transform a given region into the optimal one, and such paths may not be recovered through traditional continuous symmetrization procedures. 
	
	Therefore, we propose to utilize the flow method, which not only naturally handles continuous evolution and transforms the monotonicity investigation into the investigation of the sign of shape derivatives along the flow, but also eliminates the reliance on symmetrization and provides diverse optimization strategies. 
	
	In the present paper, we restrict the applications of the flow approach to the following two classical shape functionals: the torsional rigidity and the first eigenvalue of the Dirichlet Laplacian. Recall that the torsional rigidity of a bounded Lipschitz domain $\Omega \subset \mathbb{R}^n$ can be equivalently defined as
	\begin{align*} T(\Omega):=\int_\Omega u_\Omega\, dx,
	\end{align*}
	where $u_\Omega \in H_0^1(\Omega)$ is the weak solution of the torsion equation 
	\begin{align*}
		\begin{cases}
			-\Delta u=1\quad &\mbox{in $\Omega$,}\\
			u=0\quad &\mbox{on $\partial \Omega$.}
		\end{cases}
	\end{align*}
	The first eigenvalue of the Dirichlet Laplacian of a bounded Lipschitz domain $\Omega$ is defined as
	\begin{align*}
		\lambda_1(\Omega):=\inf\left\{\int_\Omega |\nabla v|^2: v\in H_0^1(\Omega), \int_\Omega v^2\, dx=1\right\}.
	\end{align*}
	
	The most fundamental results regarding the two functionals are the Saint-Venant inequality and the Faber-Krahn inequality, indicating that balls uniquely maximize $T(\cdot)$ while minimizing $\lambda_1(\cdot)$, among domains with fixed volume. The stability of these two inequalities is also obtained in the seminal paper \cite{B15}, where a sharp quantitative form is given. Another important topic is the discrete optimization problem, and it is conjectured by P\'olya and Szeg\"o that, among polygons with $N$ sides and fixed area, the regular polygon uniquely maximizes $T(\cdot)$ and minimizes $\lambda_1(\cdot)$. When $N=3, 4$, this is proved in \cite[Section 7.4]{PS51} via a combination of the two properties: (i) The $T(\cdot)$ is increasing and $\lambda_1(\cdot)$ is decreasing under Steiner symmetrizations; (ii) for any triangle or quadrilateral, one can find a sequence of its Steiner symmetrizations which converge to an equilateral triangle or to a square, respectively. When $N\ge 5$, a theoretical proof of the conjecture remains open, and recent advances can be seen in \cite{Bucur}, \cite{BB24}, \cite{DG} and \cite{Indrei}. The difficulty lies in the fact that the number of sides can in general increase after symmetrization, and this actually also motivates us to construct new monotonic deformation paths, which do not rely on symmetrization, to study shape optimization problems. For a comprehensive introduction of known results and methods on these two functionals, we refer to the monographs \cite{Henrot} and \cite{nshap}.
	
	In this paper, we shall apply the flow method to study the two functionals and prove new results motivated from Questions \ref{kq1}-\ref{kq2}. Our main results are described as follows.

	\subsection{Monotonicity along stretching flows and angle-bisector flows}
	
	Concerning Question \ref{kq1}, we begin by focusing on triangles, which represent an important class of planar domains. The geometry of triangles is simple, yet many profound results have been established on them, and results on triangles often inspire further important results on generic shapes.  The literature is growing very fast, and we do not attempt to list all references. For a good survey, we refer to \cite[Section 6]{nshap}.
	
	Triangles serve as an ideal model case for studying Question \ref{kq1} via the flow approach, because their concise geometric properties make it more convenient for us to choose continuous deformations that are not based on symmetrizations, but still strictly and progressively transform an arbitrary non-equilateral triangle into an equilateral one. In fact, if a given triangle $\triangle_{ABC}$ is not equilateral, then it must belong to one of the following two cases:
	
	\begin{enumerate}
		\item It has exactly one shortest height.
		\item It has exactly two shortest heights, that is, it is a sub-equilateral triangle (an isosceles triangle with the aperture less than $\pi/3)$. 
	\end{enumerate} 
	
	For the first case, a natural continuous deformation without relying on symmetrization consists in lengthening the shortest height until the triangle becomes either equilateral or sub-equilateral. Then, during the process, the triangle is intuitively closer to an equilateral shape. Concerning Question \ref{kq1} along this continuous deformation, the following theorem gives a positive answer.
	\begin{theorem}
		\label{yangsheng2}
		Let $\Omega=\triangle_{ABC}$ be a triangle, with $AB$ being the longest side and lying on the $x$-axis, and with $C$ lying on the positive $y$-axis. Suppose that $|AB|>|BC|\ge |AC|$, and let $C_t=(1+t)C$. Let $t_1>0$ be the unique number such that $|AB|=|BC_{t_1}|$. Then, for any $t\in (-1,t_1)$,
		\begin{align*}
			\frac{d}{dt}\left(\frac{T(\triangle_{ABC_t})}{|\triangle_{ABC_t}|^2}\right)>0,\quad \frac{d}{dt}\left(\lambda_1(\triangle_{ABC_t})|\triangle_{ABC_t}|\right)<0.
		\end{align*}
	\end{theorem}

	Now let $\Omega_t=\sqrt{\tfrac{|OC|}{|OC_t|}}\triangle_{ABC_t}=\triangle_{A_tB_t\tilde{C}_t}$, where $O$ is the origin. Now $\Omega_t$ has the same area of $\Omega$. Then an equivalent form of Theorem \ref{yangsheng2} states that for $t\in (-1,t_1)$,
	$\tfrac{d}{dt}T(\Omega_t)>0$, and $\tfrac{d}{dt}\lambda_1(\Omega_t)<0$. See Figure \ref{fig:ctm1} below. 
	
	For the second case, the triangle is isosceles with exactly one tallest height. The next theorem below, which is slightly more general, implies that during the process of compressing the tallest height, and before all three heights become equal, $T(\cdot)/|\cdot|^2$ is strictly increasing, and $\lambda_1(\cdot)|\cdot|$ is strictly decreasing. This was actually first proved in \cite[Theorem 1.3]{SB10} via Steiner symmetrization and continuous Steiner symmetrization argument. Our proof is more direct, and also implies a rigidity property under the transformation.
	
	\begin{theorem}
		\label{yangsheng2'}
		Let $\Omega=\triangle_{ABC}$ be a non-obtuse triangle with $AB$ being the shortest side and lying on the $x$-axis, and with $C$ lying on the positive $y$-axis. Suppose that $|AB|<|BC|\le  |AC|$, and let $C_t=(1-t)C$. Let $t_2>0$ be the unique number such that $|AB|=|BC_{t_2}|$.    Then, for $t\in (-\infty,t_2)$, \begin{align*}
			\frac{d}{dt}\left(\frac{T(\triangle_{ABC_t})}{|\triangle_{ABC_t}|^2}\right)>0,\quad \frac{d}{dt}\left(\lambda_1(\triangle_{ABC_t})|\triangle_{ABC_t}|\right)<0.
		\end{align*}
	\end{theorem}
	
	\begin{figure}[htp]
		\centering
		\begin{tikzpicture}[scale = 2.5]
			
			\fill (-0.4, 0) circle (0.02 ) node[below ] {\small$A$};
			\fill (0,1) circle (0.02 ) node[right]{\small $C_{t_1}$};
			\fill (0, 0) circle (0.02 ) node[below ] {\small$O$};
			\fill (1.2, 0) circle (0.02 ) node[below ] {\small$B$};
			\fill (0, 0.5) circle (0.02 ) node[right] {\small$C$};
			\fill (0,0.7) circle (0.02 ) node[right] {\small$C_{t}$};

			\draw[->] (-0.6,0)--(1.5,0) node[right] {$x$};
			\draw[->] (0,-0.2)--(0,1.3) node[right] {$y$};
			
			\draw (0,0)--(1.2,0);
			\draw[thick,dotted] (0,0.7)--(1.2,0);
			\draw[thick,dotted] (0,0.7)--(-0.4,0);
			\draw[thick,dotted] (0,1)--(-0.4,0);
			\draw[thick,dotted] (0,1)--(1.2,0);
			\draw (-0.4,0)--(0,0.5);
			\draw (1.2,0)--(0,0.5);

			\node at (0.3,0.2) {\small $\Omega$};
			
			\draw[->,thick] (1.2, 0.6)--(2.9,0.6);
			
			\node[above] at (2.1, 0.7) {\small $\Omega_t=\sqrt{\tfrac{|OC|}{|OC_t|}}\triangle_{ABC_t}=\triangle_{A_tB_t\tilde{C}_t}$};

			
			\draw[->] (2.9,0)--(5,0) node[right] {$x$};
			\draw[->] (3.5,-0.2)--(3.5,1.3) node[right] {$y$};
			
			\fill (3.5, 0) circle (0.02 ) node[below]{\small $O$};
			\fill (3.5-0.845*0.4,0) circle (0.02 ) node[below left]{\small $A_t$};
			\draw [thick, dotted] (3.5-0.845*0.4,0)--  (3.5, 0.7*0.845);
			\draw [thick, dotted] (3.5-0.707*0.4,0)--  (3.5, 0.707);
			\draw [thick, dotted] (3.5+0.845*1.2,0)--  (3.5, 0.7*0.845);
			\draw [thick, dotted] (3.5+0.707*1.2,0)--  (3.5, 0.707);
			\fill (3.5-0.707*0.4,0) circle (0.02 ) node[below ]{\small $A_{t_1}$};
			\fill (3.5+1.2*0.845,0) circle (0.02 ) node[below]{\small $B_{t}$};
			\fill (3.5+1.2*0.707,0) circle (0.02 ) node[below]{\small $B_{t_1}$};
			\fill (3.5, 0.7*0.845) circle (0.02 ) node[below]{\small $\tilde{C}_{t}$};
			\fill (3.5, 1*0.707) circle (0.02 ) node[right ]{\small $\tilde{C}_{t_1}$};

		\end{tikzpicture}
		\caption{Illustration of Theorem \ref{yangsheng2} for the case of $t\in (0,t_1]$. Stretching the shortest height $OC$ upward while scaling to keep the area fixed. $|\triangle_{A_tB_t\tilde{C}_t}|=|\triangle_{ABC}|$, $T(\triangle_{A_tB_t\tilde{C}_t})$ is strictly increasing, $\lambda_1(\triangle_{A_tB_t\tilde{C}_t})$ is strictly decreasing.}
		\label{fig:ctm1}
	\end{figure}

	\begin{figure}[htp]
		\centering
		\begin{tikzpicture}[scale = 2.5]
			
			\fill (-0.8, 0) circle (0.02 ) node[below ] {\small$A$};
			\fill (0,1.5) circle (0.02 ) node[right]{\small $C$};
			\fill (0, 0) circle (0.02 ) node[below ] {\small$O$};
			\fill (0.4, 0) circle (0.02 ) node[below ] {\small$B$};
			\fill (0, 1.3) circle (0.02 ) node[right] {\small$C_t$};
			\fill (0,1.1) circle (0.02 ) node[right] {\small$C_{t_2}$};

			\draw[->] (-1.1,0)--(1,0) node[right] {$x$};
			\draw[->] (0,-0.2)--(0,1.8) node[right] {$y$};
			
			\draw (-0.8,0)--(0,1.5);
			\draw[thick,dotted] (0,1.1)--(-0.8,0);
			\draw[thick,dotted] (0,1.3)--(-0.8,0);
			\draw[thick,dotted] (0,1.1)--(0.4,0);
			\draw[thick,dotted] (0,1.3)--(0.4,0);
			\draw (-0.8,0)--(0,1.5);
			\draw (0.4,0)--(0,1.5);


			\draw[->,thick] (1.0, 0.9)--(2.8,0.9);
			
			\node[above] at (1.9, 1.0) {\small $\Omega_t=\sqrt{\tfrac{|OC|}{|OC_t|}}\triangle_{ABC_t}=\triangle_{A_tB_t\tilde{C}_{t}}$};
			
			
			\fill (-0.8*1.074+4, 0) circle (0.02 ) node[below ] {\small$A_t$};
			\fill (-0.8*1.168+4, 0) circle (0.02 ) node[below left] {\small$A_{t_2}$};
			\fill (0.4*1.168+4, 0) circle (0.02 ) node[above] {\small$B_{t_2}$};
			\fill (4, 0) circle (0.02 ) node[below ] {\small$O$};
			\fill (4+0.4*1.074, 0) circle (0.02 ) node[below ] {\small$B_t$};
			\fill (4, 1.074*1.3) circle (0.02 ) node[right] {\small$\tilde{C}_t$};
			\fill (4,1.168*1.1) circle (0.02 ) node[right] {\small$\tilde{C}_{t_2}$};
			\draw [thick, dotted](-0.8*1.074+4, 0) -- (4, 1.074*1.3);
			\draw [thick, dotted](-0.8*1.168+4, 0) -- (4, 1.168*1.1);
			\draw [thick, dotted](0.4*1.074+4, 0) -- (4, 1.074*1.3);
			\draw [thick, dotted](0.4*1.168+4, 0) -- (4, 1.168*1.1);
			
			\draw[->] (2.6,0)--(4.6,0) node[right] {$x$};
			\draw[->] (4,-0.2)--(4,1.8) node[right] {$y$};

		\end{tikzpicture}
		\caption{Illustration of Theorem \ref{yangsheng2'} for the case of $0<t\le t_2$. Compressing below the tallest height $OC$ while scaling to keep the area fixed. $|\triangle_{A_tB_t\tilde{C}_t}|=|\triangle_{ABC}|$, $T(\triangle_{A_tB_t\tilde{C}_t})$ is strictly increasing and $\lambda_1(\triangle_{A_tB_t\tilde{C}_t})$ is strictly decreasing.}
		\label{fig:ctm2}
	\end{figure}

	Therefore, combining the flow processes described in Theorem \ref{yangsheng2} and Theorem \ref{yangsheng2'}, we have found a smooth deformation along which $T(\cdot)/|\cdot|^2$ is strictly increasing and $\lambda_1(\cdot)|\cdot|$ is strictly decreasing,  transforming any arbitrary non-equilateral triangle first into an isosceles triangle, and then further into an equilateral one. This continuous deformation path is new, and it takes at most two steps by stretching the shortest height and compressing the tallest height, to rapidly transform any arbitrary triangle into an equilateral triangle. This also gives another proof of the well-known fact (see \cite[Chapter VII, Page 158]{PS51}) that equilateral triangles uniquely maximize the torsional rigidity and minimize the first eigenvalue of the Dirichlet Laplacian among triangles with fixed area, without taking the limit of iterations of Steiner symmetrizations anymore.

	The flow used in the proof of Theorem \ref{yangsheng2} is generated by the following time-dependent vector field:
	\begin{align*}
		\eta(t,x,y)=\left(0,\frac{y}{1+t}\right),
	\end{align*}
	which represents the height stretching vector field. Applying a similar vector field, which generates a continuous deformation stretching one diagonal of the square, we can also prove the following new monotonicity result, which extends the result in \cite{HP} where the optimal lower bound for $\lambda_1(\cdot)|\cdot|$ on rhombuses was obtained.
	\begin{theorem}
		\label{rhombustheoremc}
		Let $\Omega_q$ be a rhombus, where $q\ge 1$ is the ratio of the lengths
		of the diagonals of a rhombus (the longest diagonal over the shortest one). Then, for $q\ge 1$, $$\frac{d}{dq}\left(T(\Omega_q)/|\Omega_q|^2\right)\le 0,\quad \frac{d}{dq}\left(\lambda_1(\Omega_q)|\Omega_q|\right)\ge 0,$$ 
		with equality holding if and only if $q=1$.
	\end{theorem}
	Theorem \ref{rhombustheoremc} states that the normalized torsional rigidity $T(\Omega_q)/|\Omega_q|^2$ is non-increasing with respect to $q$, while the normalized eigenvalue $\lambda_1(\Omega_q)|\Omega_q|$ is non-decreasing with respect to $q$. Theorem \ref{rhombustheoremc} not only allows us to compare the magnitudes of $T(\cdot)$ or $\lambda_1(\cdot)$ on two different rhombuses, but also gives a partial rigidity result that among rhombuses, squares are the unique stationary shapes along the diagonal-stretching flow. 
	
	Another reason for considering rhombuses is  their role in the symmetrization of quadrilaterals. It is well known (see for example \cite[Sec 3.3]{Henrot}) that any quadrilateral can be transformed into a rhombus through two standard steps of Steiner Symmetrization. Then, with an additional step of Steiner Symmetrization, it can be turned into a rectangle. To prove that the square uniquely maximizes $T(\cdot)$ or minimizes $\lambda_1(\cdot)$ among quadrilaterals with a given area, the standard approach involves repeatedly symmetrizing the rectangle and the rhombus to generate a sequence of shapes that progressively approximate a square. Theorem \ref{rhombustheoremc} gives a new proof of this result by demonstrating the monotonicity properties of $T(\cdot)$ and $\lambda_1(\cdot)$ through the one-step stretching or compressing deformation transforming a rhombus into a square.

	\vskip 0.2cm
	Now we continue to consider Question \ref{kq2} on triangles. First, given two non-equilateral triangles with the same area, if they share one same side, then by the continuous Steiner symmetrization, we are able to compare the values of $T(\cdot)$ and $\lambda_1(\cdot)$ on the two shapes. 
	
	It is then natural for us to consider the situation when the two given triangles share one same angle instead of one same side. Here we recall the very novel work \cite{FV19} where the following partial rigidity result on triangles is established: a triangle is isosceles with the given side as the base, if and only if the shape derivative of the transformation which rotates around the midpoint of the base and then parallel moves the base side to keep the area fixed, is zero. 
	
	Motivated by \cite{FV19} and the stretching flow approach as used in the proof of Theorem \ref{yangsheng2}, we have the following monotonicity result, which enables us to answer Question \ref{kq2} in the case of two triangles sharing one same angle. 
	
	\begin{theorem}
		\label{yangsheng1}
		Let $\Omega$ be a triangle $\triangle_{ABC}$, where $A$ lies at the origin, $B$ lies on the positive $x$-axis and $|AB|=|AC|$. Let $B_t=(1+t)B$. Then, $T(\triangle_{AB_tC})/|\triangle_{AB_tC}|^2$ is strictly decreasing, while $\lambda_1(\triangle_{AB_tC})|\triangle_{AB_tC}|$ is strictly increasing with respect to $t>0$. 
	\end{theorem}
	
	The deformations considered in Theorem \ref{yangsheng1} can generate all triangles sharing the same angle $\angle A$, as $t$ goes from $0$ to $\infty$. As a consequence, we immediately have the following:
	\begin{corollary}
		\label{jc}
		Let $\Omega_{\alpha,q}$ denote the triangle of a given area, with one angle equal to $\alpha$ and $q\ge 1$ being the ratio of the two legs near the angle. Then, for any fixed $\alpha \in (0,\pi)$, when $q\ge 1$, $T(\Omega_{\alpha,q})$ is a strictly decreasing function, and $\lambda_1(\Omega_{\alpha,q})$ is a strictly increasing function.
	\end{corollary}
	Therefore, Corollary \ref{jc} answers Question \ref{kq2} when the two given triangles share one same angle. 
	Such result was actually first established in Siudeja \cite{SB10} when the fixed angle $\alpha$ is the smallest angle. The proof there is by the ingenious use of polar symmetrization. Later, by delicate applications of the continuous Steiner symmetrization argument, Solynin \cite{Solynin20} removes the smallest assumption on the fixed angle. Though we obtained essentially the same result, our proof seems to be more natural and direct from the flow perspective. 

As a direct consequence of Theorem \ref{yangsheng1}, we obtain the following two corollaries regarding the monotonicity behavior of the shape functionals when vertices are moved along angle bisectors.
	\begin{corollary}
		\label{yangshengcor1}
		Let $\triangle ABC$ satisfy $|AB|>|AC|>|BC|$.  
		Let $A_t$ move along the internal bisector of the angle $\angle A$ toward $BC$, with $B$ and $C$ fixed, and let $t_0$ be the first time when $|AB|>|A_{t_0}C|=|BC|$.
		Then for $0<t\le t_0$, $T(\triangle_{A_tBC})/|\triangle_{A_tBC}|^2$ is a strictly increasing function, while $\lambda_1(\triangle_{A_tBC})|\triangle_{A_tBC}|$ is a strictly decreasing function. 
	\end{corollary}
	
	\begin{corollary}
		\label{yangshengcor2}
		Let $\triangle ABC$ satisfy $|AB|>|AC|\ge |BC|$.  
		Let $C_t$ move outwards along the exterior angle bisector of $\angle C$, with $A$ and $B$ fixed, and let $t_1$ be the first time when $|AB|=|AC_{t_1}|\ge |BC_{t_1}|$.
		Then for $0<t\le t_1$, $T(\triangle_{ABC_t})/|\triangle_{ABC_t}|^2$ is a strictly increasing function, while $\lambda_1(\triangle_{ABC_t})|\triangle_{ABC_t}|$ is a strictly decreasing function. 
	\end{corollary}
Corollaries \ref{yangshengcor1} and \ref{yangshengcor2} together yield another continuous deformation path via angle bisector flows, through which an arbitrary non-equilateral triangle can be transformed into an equilateral one while $T(\cdot)/|\cdot|^2$ strictly increases and $\lambda_1(\cdot)|\cdot|$ strictly decreases. 	
	
	\subsection{A new proof of the Saint-Venant inequality by the mean curvature flow}
	
	Now we address Question \ref{kq1} in the context of general bounded convex domains in $\mathbb{R}^n$. Specifically, we would like to know whether $T(\cdot)/|\cdot|^{(n+2)/n}$ increases as a shape becomes rounder.
	
	A natural way to make a smooth convex region increasingly round is to evolve its boundary by curvature flows. Two typical examples are the mean curvature flow and the inverse mean curvature flow. For a convex domain, under the mean curvature flow, boundary points with larger positive curvature move inward faster than those with smaller curvature, making the shape rounder. Under the inverse mean curvature flow, points with larger curvature move outward slower, also leading to a rounder shape. Therefore, in the spirit of Question \ref{kq1}, one might expect that shape functionals for which balls are optimal exhibit monotonicity along such flows. Indeed, several functionals have been shown to satisfy this property. For example, the isoperimetric ratio is monotone along the curve shortening flow, which evolves a curve to a round point; thus the isoperimetric inequality can be proved via the mean curvature flow (see \cite{Gage}, \cite{Grayson} and references for two dimensions, and \cite{Huisken84}, \cite{Schulze} for higher dimensions). Another example is the Minkowski inequality for total mean curvature, which is monotone along the normalized inverse mean curvature flow (see \cite{GP09}, \cite{Ger90} and \cite{HI01}).
	
	A common feature of these functional inequalities proved via curvature flows is that the functionals do not involve state functions (i.e., solutions to a PDE inside the domain). To our knowledge, no literature has proved the Saint-Venant inequality using geometric flows. This may be because geometric flows focus primarily on curvature evolution, whereas the presence of the torsion function in $T(\cdot)$ introduces nonlocal information of the evolving domains. In general, it is challenging to deduce monotonicity of shape functionals involving state functions solely from information about curvature evolution.
	
	Although we have not been able to establish the monotonicity of $T(\cdot)/|\cdot|^{(n+2)/n}$ along a geometric flow that transforms an arbitrary convex domain into a ball, we obtain the following weak monotonicity result along the mean curvature flow.
	\begin{theorem}
		\label{svinequalitynintro}
		Let $\Omega\subset\mathbb{R}^n$ be a bounded smooth convex domain. Let $\{\Omega_t\}_{t\in[0,t_0)}$ be the family of domains evolving by the mean curvature flow with initial data $\Omega_0=\Omega$, i.e., the boundary $\partial\Omega_t$ moves with normal velocity equal to its mean curvature $H$ (with respect to the inward normal). Here $t_0>0$ is the maximal time such that $\Omega_t$ shrinks to a point. Then, for any $0\le t<t_0$,
		\begin{align}
			\label{mcfweak}
			\frac{d}{dt}\left( T(\Omega_t)-\frac{T(B_1)}{|B_1|^{\frac{n+2}{n}}}|\Omega_t|^{\frac{n+2}{n}}\right)\ge 0,
		\end{align} where $B_1$ is the unit ball in $\mathbb{R}^n$.
	\end{theorem}
	
	As a consequence, we obtain a new proof of the Saint-Venant inequality and characterize the equality case in the convex smooth setting. 
	
	The following questions remain open to us:
	
	\vskip 0.2cm
	\noindent\textbf{Open Questions}: 
	Restricted to bounded smooth convex domains, is the normalized torsional rigidity $T(\cdot)/|\cdot|^{(n+2)/n}$ increasing along the mean curvature flow or the inverse mean curvature flow? Also, is $\lambda_1(\cdot)|\cdot|^{2/n}$ decreasing along these curvature flows, or is there even a geometric flow approach to proving the Faber-Krahn inequality? We also find a nonlocal flow along which the normalized torsional rigidity is increasing, see section 9, while the convergence problem remains open to us.  

\vskip 0.2cm
		We should mention that \eqref{mcfweak} is not equivalent to 
		\begin{align}
			\label{mcfstrong}
			\frac{d}{dt}\left(	\frac{T(\Omega_t)}{|\Omega_t|^{(n+2)/n}}\right)\ge 0
		\end{align}
	 along the mean curvature flow. To prove \eqref{mcfweak}, by direct computation (see section 5), it suffices to prove that
		\begin{align*}
			J(\Omega):=\frac{\int_{\partial \Omega}|\nabla u_\Omega|^2 H\, d\sigma}{|\Omega|^{\frac{2}{n}}\int_{\partial \Omega}H\, d\sigma}
		\end{align*}
		attains its maximum only at ball shape among smooth convex domains. This is easy to prove, since fixing volume, balls uniquely maximize $\int_{\partial \Omega}|\nabla u_\Omega|^2 H\, d\sigma$ while uniquely minimize $\int_{\partial \Omega}H\, d\sigma$ among smooth convex domains. However, to prove \eqref{mcfstrong}, the same strategy requires us to prove that
		\begin{align}
			\label{qfunctional}
		Q(\Omega):=\frac{|\Omega|\int_{\partial \Omega}|\nabla u_\Omega|^2 H\, d\sigma}{T(\Omega)\int_{\partial \Omega}H\, d\sigma}
		\end{align}
	 attains its supremum only when $\Omega$ is a ball. This is much more difficult to prove, since  even in two dimensions, balls maximize both the numerator and the denominator. So far we can only prove (see section 6):
	 \begin{itemize}
	 	\item $Q(\cdot)$ is uniformly bounded among smooth convex domains.
	 	\item Disks are strict local maximizers to $Q(\cdot)$ in two dimensions.
	 	\item Among planar domains enclosed by ellipses, $Q(\cdot)$ attains its maximum at balls.
	 \end{itemize}
Based on the theoretical results as well as numerical experiments, we conjecture that at least in two dimensions, $Q(\cdot)$ attains its supremum only at disks. 

We also note that a similar argument does not yield the Faber-Krahn inequality via the mean curvature flow, because the first eigenfunction of the Dirichlet Laplacian on a ball is not quadratic—a property that is essential to our proof of Theorem \ref{svinequalitynintro}.

	\subsection{On averaged squared norm of the gradient of torsion function, rigidity and monotonicity results on rectangles}
	In \cite{4L24}, it is proved by the reflection argument that for any triangle $\Omega=\triangle_{ABC}$ with $u=u_\Omega$ being its torsion function, if $|AB|\ge |BC| \ge |AC|$, then 
	\begin{align*}
		\Vert \nabla u\Vert_{L^\infty(AB)}\ge \Vert \nabla u\Vert_{L^\infty(BC)}\ge \Vert \nabla u\Vert_{L^\infty(AC)},
	\end{align*}
	and the inequalities are strict if the length comparison is strict. On the other hand, by the shape derivative of parallel movement as shown in \cite{FV19}, it is known that
	\begin{align*}
		\frac{1}{|AB|}\int_{AB}|\nabla u|^2\, ds=\frac{1}{|BC|}\int_{BC}|\nabla u|^2\, ds=\frac{1}{|AC|}\int_{AC}|\nabla u|^2\, ds.
	\end{align*}
	In other words, these results together indicate the very interesting geometric phenomenon: in an arbitrary triangular domain, the longer the side length implies a larger maximal norm of the gradient of the torsion function over the side, but no matter how long the side length is, the averaged squared norm of the gradient on each side is the same. 
	
	It is then a natural question to consider the case of rectangles. In \cite{LY24}, it is shown that if $R=\square_{ABCD}$ is a rectangle with $|AB|>|BC|$, then the torsion function $u_R$ has the property
	\begin{align*}
		\Vert \nabla u_R\Vert_{L^\infty(AB)}> \Vert \nabla u_R\Vert_{L^\infty(BC)}.
	\end{align*}
	Regarding the averaged squared norm of the gradient on each side, quite different from the triangular case, we observe the following inequality.
	
	\begin{proposition}
		\label{not-equidistribution-thmintro}
		Let $R=\square_{ABCD}$ be a rectangle with $|AB|>|AD|$. Then, 
		\begin{align}
			\label{rectangle-neumann-data}
			\frac{\int_{AB}|\nabla u_R|^{2}\,ds}{|AB|}> \frac{\int_{AD}|\nabla u_R|^{2}\,ds}{|AD|},
		\end{align}
		where $u_R$ is the torsion function in $R$. 
	\end{proposition}
	
	As a consequence of Proposition \ref{not-equidistribution-thmintro}, we have the monotonicity property of the torsional rigidity on rectangles along lengths stretching flows. 
	\begin{theorem}
		\label{recderipos}
		Let $R=\square ABCD$ be the square with vertices $A=(0,0)$, $B=(b,0)$, $C=(b,b)$ and $D=(0,b)$, $b>0$. Let $B_t=((1+t)b,0)$, $C_t=((1+t)b,b)$ and $R_t=\square AB_tC_tD$ be the rectangle with vertices $A,B_t,C_t,D$. Then, for any $t\ge 0$,
		\begin{align*}
			\frac{d}{dt}\left(\frac{T(R_t)}{|R_t|^2}\right)\le  0,
		\end{align*}
		with equality holding if and only if $t=0$.
	\end{theorem}
	
	
	The following monotonicity result is then immediately obtained as a corollary by scaling.
	\begin{proposition}
		\label{rectangle-thm}
		Let $R(s):=[0,s]\times[0,s^{-1}](s>0)$ represent a rectangle in the plane. Then the torsional rigidity $T(R(s))$ is strictly increasing with respect to $s\in (0,1)$.
	\end{proposition}

	While Proposition \ref{rectangle-thm} can also be proved by the continuous Steiner symmetrization argument (see section 8), the strict positivity of $\tfrac{d}{ds}T(R(s))$ derived as a consequence of Theorem \ref{recderipos}, cannot be simply derived from the continuous symmetrization argument. Therefore, Proposition \ref{rectangle-thm} is weaker than Theorem \ref{recderipos}.

\vskip 0.3cm		
\subsection*{Further Remarks}

We note that in recent years, there has been a growing interest in using geometric flows to study torsional rigidity from various perspectives. For instance, \cite{Hu2024} employed a Gauss curvature flow to tackle the torsional Minkowski problem, while \cite{Garcia2026} focuses on estimating and comparing the evolution of torsional rigidity under the Ricci flow and the inverse mean curvature flow on manifolds.

In contrast to those studies, our present paper focuses on monotonicity and comparison results for polygons and convex domains along novel, explicitly constructed stretching or bisector flows, as well as along the mean curvature flow, with the aim of answering Questions \ref{kq1}-\ref{kq2}. The search for alternative deformation paths is further motivated by fundamental open problems in the field, notably the unresolved Pólya–Szegő conjecture for $N$-gons as mentioned before. Since leg-stretching and angle-bisector flows can preserve the number of sides for generic $N$-gons, and by establishing monotonicity along such natural, non-symmetrizing flows for triangles (a typical case), our findings might offer some evidence in validating this longstanding conjecture.

More broadly, the many different flows studied here share an intuitive geometric feature: they progressively make a shape “more regular,” that is, closer to an optimal shape (such as an equilateral triangle). A key insight from our results is that this intuitive geometric improvement, even when achieved without symmetrization, consistently correlates with functional monotonicity. This suggests that monotonicity may be a more widespread phenomenon than previously thought, linked to a general notion of shape “betterness” rather than to the specific technique of symmetrization. Consequently, the perspective offered by Questions \ref{kq1} and \ref{kq2} can naturally stimulate researches on new geometric functional inequalities and related flow problems, as illustrated in Sections 6 and 9, which are of independent research interest.

Lastly, we note that the framework developed here is applicable beyond the problems considered. For example, a direct extension yields analogous monotonicity results for Riesz-type nonlocal energies, as treated in our forthcoming paper \cite{HLW}.

\vskip 0.3cm
	
	\noindent

	\noindent \textbf{Outline of the paper} In section 2, we present some preliminaries. In section 3, we consider height-stretching flows on triangles, and we will prove Theorem \ref{yangsheng2}, Theorem \ref{yangsheng2'} and Theorem \ref{rhombustheoremc}. In section 4, we consider leg-stretching flows and angle-bisector flows on triangles, and we will mainly prove Theorem \ref{yangsheng1} and Corollaries \ref{yangshengcor1}-\ref{yangshengcor2}. In section 5, we give a new proof of the Saint-Venant inequality for bounded smooth convex domains, via the mean curvature flow. We will mainly prove Theorem \ref{svinequalitynintro}. In section 6, we will prove some extremal properties of $Q(\cdot)$ defined in \eqref{qfunctional}. In section 7, we will prove Proposition \ref{not-equidistribution-thmintro} and Theorem \ref{recderipos}. In section 8, we use continuous Steiner symmetrization to give an alternative proof of Proposition \ref{rectangle-thm}. In section 9, we give some further remarks on Question \ref{kq1} and propose another conjecture.
	\section{Preliminaries}
	Throughout the paper, $ds$ denotes the arc-length measure on curves, and $d\sigma$ the surface measure on hypersurfaces.
	
	The following result is due to the Pohozaev's identity, see for example \cite[Theorem 2.1.8]{DVbook}.
	
	\begin{theorem}[Pohozaev's identity.]
		\label{poho}
		Let $\Omega\subset \mathbb R^n$ be a piecewise smooth bounded domain. Let $f\in C^0(\mathbb R)$ be a continuous function on the real line and let $F(z):=\int_1^z f(s)\mathrm ds$ be one of its primitives. Now let $\phi\in C^2(\Omega)\cap C^1(\bar{\Omega})$ satisfy the semi-linear boundary  value problem
		$$\begin{cases} -\Delta\phi=f(\phi) & \text{in}~\Omega \\ \phi=0 & \text{on}~\partial\Omega \end{cases}$$
		Then, the following equality holds:
		$$\left(1-\frac{n}{2}\right)\int_{\Omega}|\nabla \phi|^2\, dx+n\int_{\Omega} F(\phi)\, dx=\frac{1}{2}\int_{\partial\Omega}\,|\nabla \phi|^2 (x\cdot \nu)\,  d\sigma,$$
		where $\nu$ is the outward unit normal to $\partial\Omega$.
	\end{theorem}

	In particular, if $u$ is the torsion function on a smooth domain $\Omega$, then 
	\begin{align}
		\label{boundarytorsion}
		T(\Omega):=\int_\Omega u\, dx=\frac{1}{n+2}\int_{\partial \Omega}|\nabla u(x)|^2 (x\cdot \nu) \, d\sigma.
	\end{align}
	
	If $v$ is the first eigenfunction of the Dirichlet Laplacian with $\int_\Omega v^2\, dx=1$, then 
	\begin{align}
		\label{boundarylambda1}
		\lambda_1(\Omega)=\frac{1}{2}\int_{\partial \Omega}|\nabla v(x)|^2 (x\cdot \nu) \, d\sigma.    
	\end{align}
By boundary regularity results near convex vertices (see \cite{Gri}), \eqref{boundarytorsion}-\eqref{boundarylambda1} also hold if $\Omega$ is a convex polygonal domain. These two formulas will be frequently used in this paper, as they translate integrals over the domain to boundary integrals.
	\vskip 0.2cm
	
	The next result, which is well-known and can be obtained by the moving plane method, is also often used in this paper to determine the sign of derivatives of $T(\cdot)$ and $\lambda_1(\cdot)$ along suitable flows. It is also used in \cite{FV19} to obtain rigidity results.
	
	\begin{theorem}
		\label{symmetrycomparison}
		Let $u$ satisfy
		\begin{align*}
			\begin{cases}
				-\Delta u=f(u)\quad &\mbox{in $\Omega$}\\
				u=0\quad &\mbox{on $\partial \Omega$}
			\end{cases}
		\end{align*}
	where $f:\mathbb{R}\to\mathbb{R}$ is a positive $C^1$ function, and $\Omega$ is a triangle, denoted by $\triangle_{ABC}$. Let $M$ be the midpoint of $AB$. For any $P\in int(BM)$, let $P'\in AM$ such that $|PM|=|MP'|$. If $|BC|>|AC|$, then $|\nabla u|(P')>|\nabla u|(P)$.   
	\end{theorem}
	
	\begin{figure}[htp]
		\centering
		\begin{tikzpicture}[scale = 2.5]
			\pgfmathsetmacro\L{1/sqrt(3)};
			\pgfmathsetmacro\h{1};
			\pgfmathsetmacro\t{0.4};

			\draw[->] (-0.8,0)--(1.5,0) node[right] {$x$};
			\fill[gray, yellow] (0.2, 0) -- (0.2, 0.785365) -- (\L+\t, 0) -- cycle;
			\fill[gray, green] (0.2, 0) -- (0.2, 0.785365) -- (-\L, 0) -- cycle;
			
			\draw (-\L,0)--(\L+\t,0);
			\draw (\L+\t,0)--(0,\h);
			\draw (0,\h)--(-\L,0);
			\draw[thick,dotted] (0.2,0)--(0.2,1.3);
			\draw[thick,dotted] (0.2,0.785365)--(-\L,0);

			\draw[->] (0,-0.2)--(0,1.3) node[right] {$y$};
			
			\fill (-\L, 0) circle (0.02 ) node[below ] {\small$A$};
			\fill (-\L+\t, 0) circle (0.02 ) node[below ] {\small$P'$};
			\fill (\L, 0) circle (0.02 ) node[below ] {\small$P$};
			\fill (\L+\t, 0) circle (0.02 ) node[below right] {\small$B$};
			\fill (0, \h) circle (0.02 ) node[above right] {\small$C$};
			\fill (0.2, 0) circle (0.02 ) node[below] {\small$M$};
			\fill (0.2, 0.785365) circle (0.02 ) node[above right] {\small$D$};
			
			
		\end{tikzpicture}
		\caption{Comparison of the gradient norms at two points}
		\label{fig:ctm''}
	\end{figure}
	
	\begin{proof}
		Let $D$ be the intersection point of the perpendicular bisector of AB and the side BC. Since $|BC|>|AC|$, the reflection of $\triangle BDM$ about $DM$ strictly lies inside $\triangle ABC$ (see Figure \ref{fig:ctm''}). Therefore, it is standard by the moving plane method (see for example \cite{BN}, and similar argument is used in \cite{LY24}, \cite{4L24}, etc) and the Hopf lemma to conclude that $|\nabla u|(P')>|\nabla u|(P)$.
	\end{proof}
	
	The next theorem derives the evolution equation for the torsional rigidity and first eigenvalue of the Dirichlet Laplacian along a flow map generated by a smooth time-dependent vector field. Classical shape derivative formula along a flow is often presented at a particular time, say $t=0$, see for example \cite[Section 2.5]{Henrot} and \cite[Chapter 5]{HP05}, while the formula of the whole-time evolution equation along a flow is less commonly written in literature. Therefore, we also present the proof.

	\begin{theorem}
		\label{Tshapederivative}
		Let $\Omega$ be a smooth bounded domain in $\mathbb{R}^n$. Let $\eta(t,x): \mathbb R_+\times \mathbb{R}^n \to \mathbb{R}^n$ be a $C^\infty$ vector field with compact support in $\mathbb{R}^n$ (or more generally, satisfying a global Lipschitz condition with respect to $x$ and uniformly for $t$), and let $F_t(x):=F(t,x)$ be the flow map generated by $\eta$, i.e.,
		\begin{align*}
			\begin{cases}
				\frac{\partial}{\partial t}F(t,x)=\eta(t,F(t,x)), \quad t>0,\\
				F(0,x)=x.
			\end{cases}
		\end{align*}
		Then $F_t \in C^1(\mathbb{R}_+ \times \mathbb{R}^n; \mathbb{R}^n)$ is a diffeomorphism for each $t \ge 0$. Also, for any $t>0$, the shape derivatives of the torsional rigidity and the first eigenvalue exist and are given by
		\begin{align}
			\label{sikong1}
			\frac{d}{dt}T(F_t(\Omega)) &= \int_{\partial F_t(\Omega)}|\nabla u(t)|^2 \, \eta(t,x)\cdot \nu\, d\sigma, \\
			\label{sikong2}
			\frac{d}{dt}\lambda_1(F_t(\Omega)) &= -\int_{\partial F_t(\Omega)}|\nabla v(t)|^2 \, \eta(t,x)\cdot \nu\, d\sigma,
		\end{align}
		where:
		\begin{itemize}
			\item $u(t)$ is the torsion function on $F_t(\Omega)$, i.e., the unique solution to $-\Delta u = 1$ in $F_t(\Omega)$ with $u=0$ on $\partial F_t(\Omega)$.
			\item $v(t)$ is the first (positive) $L^2$-normalized eigenfunction on $F_t(\Omega)$, i.e., satisfying $-\Delta v = \lambda_1 v$ in $F_t(\Omega)$, $v=0$ on $\partial F_t(\Omega)$, and $\int_{F_t(\Omega)} v(t)^2\, dx = 1$.
			\item $\nu$ is the outward unit normal vector on $\partial F_t(\Omega)$.
		\end{itemize}
		The above results also hold if $F_t(\Omega),\, t\ge 0$ are all convex polygons.
	\end{theorem}
	
	\begin{proof}
		Since $\eta$ has uniform global Lipschitz condition, by ODE theory, $F_t: \mathbb{R}^n \to \mathbb{R}^n$ is a smooth diffeomorphism for each $t$. We define the pullback of functions to the fixed domain $\Omega$. For the torsion function, let
		\[
		\tilde{u}(t, x) = u(t)(F_t(x)), \quad x \in \Omega.
		\]
		The function $\tilde{u}(t, \cdot) \in H^1_0(\Omega)$ satisfies a perturbed Poisson equation on $\Omega$ whose coefficients depend smoothly on $t$ via the deformation gradient $DF_t$. Standard elliptic theory with parameters implies that the mapping $t \mapsto \tilde{u}(t, \cdot)$ is differentiable from $\mathbb{R}_+$ into $H^1_0(\Omega)$ (see, e.g., the framework in \cite{HP05}).
		
	 The shape derivative $u'(t) \in H^1(F_t(\Omega))$ is then defined by the relation
		\begin{align}
			\label{relation}
			u'(t)(F_t(x)) := \frac{d}{dt}\tilde{u}(t,x) - \nabla u(t)(F_t(x)) \cdot \eta(t, F_t(x)).
		\end{align}
		A crucial point is the regularity of $u(t)$ on the boundary.
		\begin{itemize}
			\item If $\Omega$ is smooth, then $F_t(\Omega)$ is smooth, and standard elliptic regularity yields $u(t) \in C^\infty(\overline{F_t(\Omega)})$, and hence $|\nabla u(t)|^2$ is continuous on the boundary.
			\item If $F_t(\Omega)$ is a convex polygon in $\mathbb{R}^2$, then from the seminal monograph \cite{Gri}, $u(t)\in H^2(F_t(\Omega))$. This also implies that $\nabla u(t)$ has a well-defined trace in $L^2(\partial F_t(\Omega))$, and the boundary integral in \eqref{sikong1} is meaningful.
		\end{itemize}
		The same regularity considerations hold for the eigenfunction $v(t)$, since $\lambda_1$ is simple.

		Now we prove \eqref{sikong1}. From integration by parts, we have
		\begin{align} 
			\label{fuxie0}T(F_t(\Omega))=\int_{F_t(\Omega)}2u(t)\, dx-\int_{F_t(\Omega)}|\nabla u(t)|^2\, dx.
		\end{align}
		
		Using the well-known formula (see for example \cite[Chapter 5]{HP05}):
		\begin{align*}
			\frac{\partial}{\partial t}\int_{F_t(\Omega)}f(t,x)\, dx=\int_{F_t(\Omega)}f_t(t,x)\, dx+\int_{\partial F_t(\Omega)} f(t,x)\eta(t,x)\cdot \nu \, d\sigma,
		\end{align*}
		we have
		\begin{align}
			\label{fuxie1}
			\frac{d}{d t}\int_{F_t(\Omega)}2u(t)\, dx=\int_{F_t(\Omega)}2u'(t)\, dx
		\end{align}and
		\begin{align*}
			&\frac{d}{dt}\int_{F_t(\Omega)}|\nabla u(t)|^2\, dx\\
			=&\int_{F_t(\Omega)}2\nabla u(t)\nabla u'(t)\, dx+\int_{\partial F_t(\Omega)}|\nabla u(t)|^2 \eta(t,x)\cdot \nu\, d\sigma\\
			=& \int_{F_t(\Omega)} 2(-\Delta u(t)) u'(t)\, dx+\int_{\partial F_t(\Omega)}2u'(t) \partial_\nu u(t)\, d\sigma+\int_{\partial F_t(\Omega)}|\nabla u(t)|^2 \eta(t,x)\cdot \nu\, d\sigma.
		\end{align*}
		Since $-\Delta u(t)=1$ in $F_t(\Omega)$ and 
		\begin{align*}
			u(t,F_t(x))=0 \quad \mbox{on $\partial \Omega$} \Rightarrow u'(t)+  (\eta\cdot \nu)\partial_\nu u(t)=0\quad \mbox{on $\partial F_t(\Omega)$},
		\end{align*}
		we have
		\begin{align}
			\label{fuxie2}
			\frac{d}{dt}\int_{F_t(\Omega)}|\nabla u(t)|^2\, dx=\int_{F_t(\Omega)} 2u'(t)\, dx-\int_{\partial F_t(\Omega)}|\nabla u(t)|^2 \eta(t,x)\cdot \nu\, d\sigma.
		\end{align}
		By \eqref{fuxie0}-\eqref{fuxie2}, we obtain the desired formula for the shape derivative of torsional rigidity.
		
		Next, we prove \eqref{sikong2}. Since
		\begin{align*}
			\int_{F_t(\Omega)} v(t)^2\, dx=1,
		\end{align*}
		by taking the derivative we have
		\begin{align}
			\label{0lambdaconstriant}
			0= \int_{F_t(\Omega)} 2v(t)v'(t)\, dx+\int_{\partial F_t(\Omega)}v(t)^2 \eta(t,x)\cdot \nu\, d\sigma=\int_{F_t(\Omega)} 2v(t)v'(t)\, dx,
		\end{align}where $v'$ is the shape derivative of $v$.
		Using that on $\partial F_t(\Omega)$, 
		\begin{align}
			\label{bdv}
			v'(t)=-(\eta\cdot \nu)\partial_\nu v(t),
		\end{align}
		and by \eqref{0lambdaconstriant}, we have
		\begin{align*}
			\frac{d}{dt}\lambda_1(F_t(\Omega))=&\frac{d}{dt}\int_{F_t(\Omega)}|\nabla v(t)|^2\, dx\\
			=& \int_{F_t(\Omega)}2\nabla v(t) \cdot \nabla v'(t)\, dx+\int_{\partial F_t(\Omega)}|\nabla v(t)|^2\eta\cdot \nu\, d\sigma\\
			=&\lambda_1(F_t(\Omega))\int_{F_t(\Omega)}v(t) v'(t)\, dx+\int_{\partial F_t(\Omega)}2\partial_\nu v(t) v'(t)\, d\sigma+\int_{\partial F_t(\Omega)}|\nabla v(t)|^2\eta\cdot \nu\, d\sigma\\
			=& -\int_{\partial F_t(\Omega)}|\nabla v(t)|^2 \eta(t,x)\cdot \nu\, d\sigma.
		\end{align*}
		\end{proof}
	
	This theorem tells us that the derivatives of $T(\cdot)$ and $\lambda_1(\cdot)$ along a flow only depend on the normal component of the velocity field restricted on the boundary of the domain.
	
	\section{Monotonicity under the shortest height stretching flow and tallest height compressing flow}
	In this section, we prove Theorem \ref{yangsheng2}, Theorem \ref{yangsheng2'} and Theorem \ref{rhombustheoremc}.

	\begin{proof}[Proof of Theorem \ref{yangsheng2}]
		Without loss of generality, we may assume that $|CO|=1$. Let $\Delta_t=\triangle ABC_t$, $$F_t(x,y)=(x,(1+t)y)$$ and$$\eta(t,x,y)=\left(0,\frac{y}{1+t}\right).$$
		Then, $F_t(\Omega)=\Delta_t$, $\Omega_t=\Delta_t/\sqrt{1+t}$ and
		\begin{align*}
			\frac{\partial}{\partial t}F_t(x,y)=\eta\left(t,F_t(x,y)\right).
		\end{align*}
		To prove Theorem \ref{yangsheng2}, it suffices to prove $\tfrac{d}{dt}T(\Omega_t)>0$ and $\tfrac{d}{dt}\lambda_1(\Omega_t)<0$ for $t\in (-1,t_1)$. Let $u(t)=u_{\Delta_t}$ be the torsion function on $\Delta_t$. Then, by Theorem \ref{poho} and Theorem \ref{Tshapederivative}, we have
		\begin{align*}
			\frac{d}{dt}T(\Omega_t)=&\frac{d}{dt}\left(\frac{1}{(1+t)^2}T(F_t(\Omega))\right)\\
			=&-2\frac{1}{(1+t)^3}T(\Delta_t)+\frac{1}{(1+t)^2}\int_{\partial \Delta_t} |\nabla u_{\Delta_t}|^2 \eta(t,x,y)\cdot \nu\, ds\\
			=& -\frac{1}{2(1+t)^3}\int_{\partial \Delta_t} |\nabla u(t)(x,y)|^2 (x,y)\cdot \nu\, ds+\frac{1}{(1+t)^3}\int_{\partial \Delta_t} |\nabla u(t)(x,y)|^2 (0,y)\cdot \nu\, ds\\
			=& \frac{1}{2(1+t)^3}\int_{\partial \Delta_t} |\nabla u(t)(x,y)|^2 (-x,y)\cdot \nu \, ds.
		\end{align*}
		Let $\beta_t=\angle ABC_t$ and $\gamma_t=\angle BAC_t$, as illustrated in Figure \ref{fig:add1}. Here and througout the paper, $\angle$ means the angle.
		\begin{figure}[htp]
			\centering
			\begin{tikzpicture}[scale = 1.5]
				\pgfmathsetmacro\d{180/pi};

				\fill[green,opacity=0.4] (-1,0)--(3/2,0)--(-1/2,1)--cycle;
				\fill[yellow,opacity=0.4] (0,2)--(3/2,0)--(-1/2,1)--cycle;
				\draw (-1,0)--(0,2); \draw (0,2)--(2,0);\draw (-1,0)--(2,0);
				
				\draw[->] (-2,0)--(3,0) node[below right] {$x$};
				\draw[->] (0,-1)--(0,3) node[above left] {$y$};
				
				\node[below left] at (-1,0) {\small $A$};
				\node[above left] at (0,2) {\small $C_{t}$};
				\node[below right] at (2,0) {\small $B$};
				\node[below right] at (0,0) {\small $O$};
				
				\fill (-1,0) circle (0.02);
				\fill (0,2) circle (0.02);
				\fill (2,0) circle (0.02);
				
				\draw[samples=500,black, domain=0:pi/3, variable=\t] plot ({-1+0.2*cos(\t*\d)},{0.2*sin(\t*\d)});
				\node at (-0.7,0.15) {$\gamma_{t}$};
				
				\draw[samples=500,black, domain=15*pi/20:pi, variable=\t] plot ({2+0.2*cos(\t*\d)},{0.2*sin(\t*\d)});
				\node at (1.6,0.15) {$\beta_{t}$};
				
				\fill (1,1) circle (0.03) node[right] {$M_{t}$};
				\draw[samples=500,thick,dotted, domain=-1:2, variable=\t] plot ({\t},{1*\t});
				\node[right] at (2,1.8) {\small $y-(\tan \beta_{t})x=0$};
				
				\fill (-1/2,1) circle (0.03) node[left] {$N_{t}$};
				\draw[samples=500,thick,dotted, domain=-1:0.5, variable=\t] plot ({\t},{-2*\t});
				\node[left] at (-1,1) {\small $y+(\tan \gamma_{t})x=0$};
				
				\fill (0.5,1.5) circle (0.02) node[right] {$p_{t}'$};
				\fill (1.5,.5) circle (0.02) node[right] {$p_{t}$};
				
				\fill (-0.3,1.4) circle (0.02) node[left] {$q_{t}'$};
				\fill (-0.7,0.6) circle (0.02) node[left] {$q_{t}$};

			\end{tikzpicture}
			\caption{For $-1<t<t_1$, $|AB|>|BC_t|\ge |AC_t|$. $y=x\tan \beta_t$ and $y=-x\tan \gamma_t$ pass through the midpoints of $BC_t$ and $AB_t$, respectively. $|p_tM_t|=|p_t'M_t|$, $|\nabla u(t)(p_t)|<|\nabla u(t)(p_t')|$, $|q_tN_t|=|q_t'N_t|$, $|\nabla u(t)(q_t)|<|\nabla u(t)(q_t')|$.}
			\label{fig:add1}
		\end{figure}
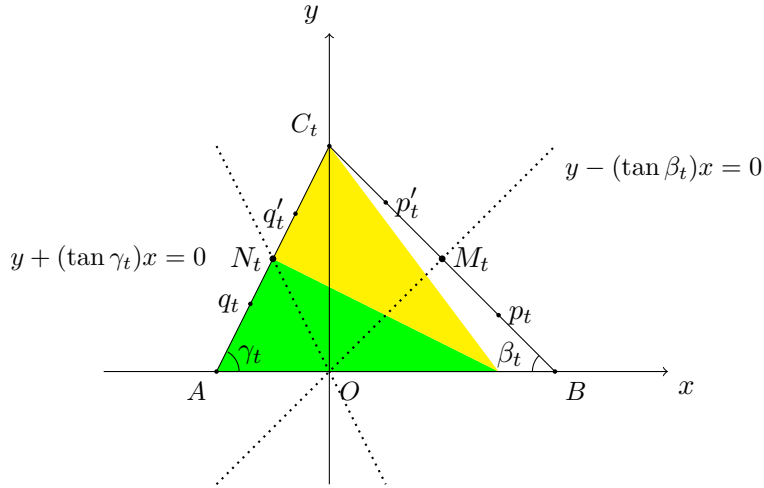
		
On $BC_t$, we have
		\begin{align*}
			(-x,y)\cdot (\sin\beta_t,\cos\beta_t)=\cos \beta_t (y-x\tan\beta_t).
		\end{align*}
		On $AC_t$, we have
		\begin{align*}
			(-x,y)\cdot (-\sin\gamma_t,\cos\gamma_t)=\cos \gamma_t (y+x \tan \gamma_t).
		\end{align*}
		Hence
		\begin{align*}
			&\frac{d}{dt}T(\Omega_t)\\
			=&\frac{\cos \beta_t}{2(1+t)^3}\int_{BC_t} |\nabla u(t)(x,y)|^2 (y-x\tan \beta_t) \, ds+\frac{\cos \gamma_t}{2(1+t)^3}\int_{AC_t} |\nabla u(t)(x,y)|^2 (y+x\tan \gamma_t) \, ds.
		\end{align*}
		Let $M_t$ be the midpoint of $BC_t$. Since $y=(\tan \beta_t)x$ is exactly the line passing through the origin and $M_t$, for any point $p_t=(x,y)\in BM_t$ and $p_t'=(x',y')\in M_tC_t$ with $|p_tM_t|=|p_t'M_t|$, we have 
		\begin{align*}
			0>y-x\tan\beta_t=-(y'-x'\tan\beta_t).
		\end{align*}
		Also, for $-1<t<t_1$, $|AB|>|AC_t|$, thus by Theorem \ref{symmetrycomparison}, $|\nabla u(t)|(p_t)<|\nabla u(t)|(p_t')$.
		Hence 
		\begin{align*}
			\int_{BC_t} |\nabla u(t)(x,y)|^2 (y-x\tan \beta_t) \, ds=\int_{BM_t}\left(|\nabla u(t)(p_t)|^2-|\nabla u(t)(p_t')|^2\right)(y-x\tan \beta_t)\, ds>0.
		\end{align*}
		Similarly, let $N_t$ be the midpoint of side $AC_t$. Since $y=-(\tan \gamma_t)x$ is the line passing through the origin and $N_t$, for any point $q_t=(x,y)\in AN_t$ and $q_t'=(x',y')\in N_tC_t$ with $|q_tN_t|=|q_t'N_t|$, we have 
		\begin{align*}
			0>y+x\tan\gamma_t=-(y'+x'\tan\gamma_t).
		\end{align*}
		Again by Theorem \ref{symmetrycomparison}, for any $-1<t<t_1$,
		\begin{align*}
			\int_{AC_t} |\nabla u(t)(x,y)|^2 (y+x\tan \gamma_t) \, ds=\int_{AN_t}\left(|\nabla u(t)(q_t)|^2-|\nabla u(t)(q_t')|^2\right)(y+x\tan \gamma_t)\, ds>0.
		\end{align*}
		Hence $\tfrac{d}{dt}T(\Omega_t)>0$ for any $t\in (-1,t_1)$. 
		
		That $\tfrac{d}{dt}\lambda_1(\Omega_t)<0$ for $t\in (-1,t_1)$ can be also derived in a similar way, thanks to Theorems \ref{poho}-\ref{Tshapederivative}. (See also the proof of Theorem \ref{yangsheng1} where we write the full details.) Therefore, we finish the proof of Theorem \ref{yangsheng2}.
	\end{proof}

	Next, we prove Theorem \ref{yangsheng2'}.

	\begin{proof}[Proof of Theorem \ref{yangsheng2'}]
		Without loss of generality, we assume that $|OC|=1$. Let 
		\begin{align*}
			F_t(x,y)=\left(x,(1-t)y\right),\quad \eta(t,x,y)=\left(0,-\frac{y}{1-t}\right).
		\end{align*}
		Then $F_t$ is generated by the vector field $\eta$, $\triangle_{ABC_t}=F_t(\triangle_{ABC})$ and we let 
		\begin{align*}
			\Omega_t=\frac{1}{\sqrt{1-t}}F_t(\Omega).
		\end{align*}
		Let $u(t)=u_{\Omega_t}$ be the torsion function on $\Omega_t$. By the similar calculation of the shape derivative as in the proof of Theorem \ref{yangsheng2}, we have
		\begin{align*}
			\frac{d}{dt}T(\Omega_t)=&\frac{\cos \beta_t}{2(1-t)^3}\int_{BC_t} |\nabla u(t)(x,y)|^2(x\tan\beta_t-y)\, ds\\
			&+\frac{\cos \gamma_t}{2(1-t)^3}\int_{ AC_t} |\nabla u(t)(x,y)|^2(-x\tan\gamma_t-y)\, ds,
		\end{align*}
		where $\beta_t=\angle C_tBA$ and $\gamma_t=\angle C_t AB$. The similar argument in the proof of Theorem \ref{yangsheng2} ensures that $\tfrac{d}{dt}T(\Omega_t)>0$, for any $t\in (-\infty, t_2)$. This is equivalent to 
		\begin{align*}
			\frac{d}{dt}\left(\frac{T(\triangle_{ABC_t})}{|\triangle_{ABC_t}|^2}\right)>0.
		\end{align*}
		
		Thanks to \eqref{boundarylambda1}, the case of $\lambda_1(\cdot)$ is similar (see also the proof of Theorem \ref{yangsheng1}) and we omit the proof.
	\end{proof} 
	
	\vskip 0.2cm
	Applying the same stretching flow argument, we can obtain the monotonicity as well as rigidity result on rhombuses.
	
	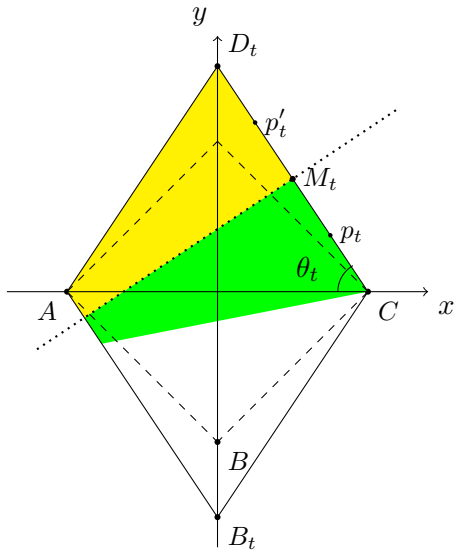
\begin{figure}[htp]
		\centering
		\begin{tikzpicture}[scale = 2]
			\pgfmathsetmacro\a{1};
			\pgfmathsetmacro\t{0.5};
			\pgfmathsetmacro\d{180/pi};
			\pgfmathsetmacro\at{(1+\t)*\a};
			\pgfmathsetmacro\xt{1/(\a/\at+\at/\a)*(\a*\a/(2*\at)-\at-\at/2)};
			\pgfmathsetmacro\yt{-\at/\a*(\xt+\a)};
			\pgfmathsetmacro\rx{1/(\a/(2*\at)+\at/(2*\a))*(-\at+\a*\a/\at)};
			\pgfmathsetmacro\ry{-\at/\a*(\rx+\a)};

			\fill (-\a,0) circle(0.02) node[below left] {\small $A$};
			\fill (\a,0) circle(0.02) node[below right] {\small $C$};
			\fill (0,-\a) circle(0.02) node[below right] {\small $B$};
			\fill (0,\a) circle(0.02) node[above right] {\small $D$};
			
			\fill[yellow,opacity=0.4] (\xt,\yt)--(-\a,0)--(0,\at)--(\a/2,\at/2)--cycle;
			\fill[green,opacity=0.4] (\xt,\yt)--(\rx,\ry)--(\a,0)--(\a/2,\at/2)--cycle;
			\draw (-\a,0)--(0, \at);
			\draw (\a,0)--(0, \at);
			\draw (-\a,0)--(0,-\at);
			\draw (\a,0)--(0,-\at);
			
			\draw[->] (-1.4*\a,0)--(1.4*\a,0) node[below right] {$x$};
			\draw[->] (0,-1.7*\a)--(0,1.7*\a) node[above left] {$y$};
			
			\draw[dashed] (-\a,0)--(0, \a);
			\draw[dashed] (\a,0)--(0, \a);
			\draw[dashed] (-\a,0)--(0,-\a);
			\draw[dashed] (\a,0)--(0,-\a);
			
			\fill (0,-\at) circle(0.02) node[below right] {\small $B_{t}$};
			\fill (0,\at) circle(0.02) node[above right] {\small $D_{t}$};
			\fill (\a/2,\at/2) circle(0.02) node[right] {\small $M_{t}$};
			
			\fill (\a/4,3*\at/4) circle(0.015) node[right] {\small $p'_{t}$};
			\fill (3*\a/4,\at/4) circle(0.015) node[right] {\small $p_{t}$};

			\draw[samples=500,thick,black,dotted, domain=-1.2*\a:1.2*\a, variable=\s] plot ({\s},{\a/\at*(\s-\a/2)+\at/2});

			\draw[samples=500,black, domain=2*pi/3:pi, variable=\t] plot ({\a+0.2*cos(\t*\d)},{0.2*sin(\t*\d)});
			\node at (0.6*\a,0.15*\a) {$\theta_{t}$};

		\end{tikzpicture}
		\caption{Stretch the diagonal $BD$ of the square $\square ABCD$. During the stretching process, $T(\cdot)/|\cdot|^2$ is decreasing, while $\lambda_1(\cdot)|\cdot|$ is increasing.}
		\label{fig:ctmrhom}
	\end{figure}

	\begin{proof}[Proof of Theorem \ref{rhombustheoremc}]
		Since the functionals $T$ and $\lambda_1$ are invariant under rotation, we may assume without loss of generality that the rhombuses $\Omega_q$ have vertices $A=(-a,0)$, $B_t=(0,-a(1+t))$, $C=(a,0)$, $D_t=(0,a(1+t))$, where $a>0$ and $t=q-1$.

		Let $F_t(x,y)=(x,(1+t)y)$ and 
		\begin{align*}
			\eta(t,x,y)=\left(0,\frac{y}{1+t}\right).
		\end{align*}
		Then $F_t$ is the flow generated by $\eta$. 
		Let $u(t)$ be the torsion function on $\Omega_t$, and thus $u(t)$ is axially symmetric. By \eqref{boundarytorsion} and symmetry, we have
		\begin{align*}
			T(\Omega_t)=\int_{CD_t}|\nabla u(t)|^2 (x,y)\cdot (\sin \theta_t,\cos\theta_t)\, ds,
		\end{align*}
		where $\theta_t=\angle ACD_t$.
		
		Hence by \eqref{boundarytorsion} and Theorem \ref{Tshapederivative}, we have
		\begin{align*}
			\frac{d}{dt}\left(\frac{T(\Omega_t)}{|\Omega_t|^2}\right)=&\frac{1}{|\Omega_0|^2}\frac{d}{dt}\left((1+t)^{-2}T(F_t(\Omega_0))\right)\\
			=&\frac{1}{|\Omega_0|^2}\cdot (-2)(1+t)^{-3}\int_{CD_t}|\nabla u(t)(x,y)|^2 (x,y)\cdot (\sin \theta_t,\cos \theta_t)\, ds\\
			&+\frac{1}{|\Omega_0|^2}\cdot (1+t)^{-2}\cdot 4\int_{CD_t}|\nabla u(t)(x,y)|^2 \left(0,\frac{y}{1+t}\right)\cdot (\sin \theta_t,\cos \theta_t)\, ds\\
			=& \frac{2}{|\Omega_0|^2 (1+t)^3}\cos \theta_t\int_{CD_t}|\nabla u(t)(x,y)|^2(y-x\tan\theta_t)\, ds.
		\end{align*}
		
		Let $M_t$ be the midpoint of $CD_t$. For any $p_t=(x,y)\in CM_t$ and $p'_t=(x',y')\in M_tD_t$ with $|p_tM_t|=|p_t'M_t|$, we have by the geometry of the 
		rhombus and the reflection method (similar to the argument in \cite[Propositions 3.1-3.2]{4L24} and \cite[Theorem 5.1]{LY24}) to derive
		\begin{align*}
			|\nabla u(t)(p_t)|\ge |\nabla u(t)(p'_t)|,
		\end{align*}
		with $"="$ holding if and only if $t=0$.
		
		Also, $-(y'-x'\tan \theta)=y-x\tan \theta<0$. Hence
		\begin{align*}
			\frac{d}{dt}\left(\frac{T(\Omega_t)}{|\Omega_t|^2}\right)=\frac{2\cos \theta_t}{|\Omega_0|^2 (1+t)^3}\int_{CM_t}\left(|\nabla u(t)(p_t)|^2-|\nabla u(t)(p_t')|^2\right)(y-x\tan\theta_t)\, ds\le 0,
		\end{align*}
		with the equality holding if and only if $t=0$.
		
		The proof of monotonicity of $\lambda_1(\Omega_t)|\Omega_t|$ is similar (see also the proof of Theorem \ref{yangsheng1}), so we omit it. 
	\end{proof}

	\section{monotonicity along leg-stretching flows and angle-bisector flows}
	
	In this section, we shall prove Theorem \ref{yangsheng1} and its consequences.
	
	\begin{proof}[Proof of Theorem \ref{yangsheng1}]
	Without loss of generality, we may assume that $\Omega=\triangle_{ABC}$ with $|AB|=|AC|=1$, $A=(0,0)$ and $B=(1,0)$. Let
	\begin{align*}
		\Omega_t=\frac{1}{\sqrt{1+t}}\triangle_{AB_tC},
	\end{align*}
	and thus $|\Omega_t|=|\Omega|$. Then, to prove Theorem \ref{yangsheng1}, it suffices to prove the monotonicity of $T(\Omega_t)$ and $\lambda_1(\Omega_t)$. For an illustration, see Figure \ref{fig:ctm} below.

	\begin{figure}[htp]
		\centering
		\begin{tikzpicture}[scale = 2.5]
			
			\fill (0, 0) circle (0.02 ) node[below ] {\small$A$};
			\fill (1, 0) circle (0.02 ) node[below ] {\small$B$};
			\fill (0.707, 0.707) circle (0.02 ) node[above right] {\small$C$};
			\fill (1.2, 0) circle (0.02 ) node[below] {\small$B_{t}$};

			\draw[->] (-0.3,0)--(1.5,0) node[right] {$x$};
			\draw[->] (0,-0.2)--(0,1.1) node[right] {$y$};
			
			\draw (0,0)--(1.2,0);
			\draw[thick,dotted] (1,0)--(0.707,0.707);
			\draw (0.707,0.707)--(1.2,0);
			\draw (0.707,0.707)--(0,0);

			\node at (0.5,0.2) {\small $\Omega$};
			
			\draw[->,thick] (1.6, 0.4)--(2.6,0.4);
			
			\node[above] at (2.1, 0.5) {\small $\Omega_t=\sqrt{\tfrac{|AB|}{|AB_t|}}\triangle_{AB_tC}$};
			
			
			\draw[->] (2.7,0)--(4.5,0) node[right] {$x$};
			\draw[->] (3,-0.2)--(3,1.1) node[right] {$y$};
			
			\draw (3,0)--(3+1.2,0);
			\draw (3+1.2,0)--(3+0.707,0.707);
			\draw (3,0)--(3+0.833*0.707,0.833*0.707);
			\draw [thick,dotted](3+0.833*0.707,0.833*0.707)--(3+0.833*1.2,0);
			\draw (3+0.707,0.707)--(3,0);
			
			\fill (3, 0) circle (0.02 ) node[below ] {\small$A$};
			\fill (4.2, 0) circle (0.02 ) node[below ] {\small$B_t$};
			\fill (3+0.707, 0.707) circle (0.02 ) node[above right] {\small$C$};
			
			\node at (3.5,0.2) {\small $\Omega_t$};
			
		\end{tikzpicture}
		\caption{Stretching one leg and scaling to fix the area}
		\label{fig:ctm}
	\end{figure}

		Let $\alpha=\angle A$ and $\beta_t=\angle AB_tC$. 
		We also let 
		\begin{align*}
			F_t(x,y)=\left((1+t)x-t(\cot \alpha)y,y \right)
		\end{align*}
		and
		\begin{align*}
			\eta(t,x,y)=\left(\frac{x-y\cot\alpha}{1+t},0\right).
		\end{align*}
		Hence $F_t$ maps $\triangle_{ABC}$ to $\triangle_{AB_tC}$, and
		\begin{align*}
			\frac{\partial}{\partial t}F_t(x,y)=\eta\left(t,F_t(x,y)\right).
		\end{align*}
		Let $D_t=\triangle_{AB_tC}$ and $u(t)$ be the torsion function on $D_t$. Since
		\begin{align*}
			T(\Omega_t)=\frac{1}{(1+t)^2}T(D_t),
		\end{align*}
		by Theorem \ref{Tshapederivative}, we have
		\begin{align*}
			\frac{d}{dt}T(\Omega_t)=-2(1+t)^{-3}T(D_t)+(1+t)^{-2}\int_{\partial D_t}|\nabla u(t)(x,y)|^2 \eta(t,x,y)\cdot \nu\, ds.
		\end{align*}
		Note that on $AB_t\cup AC$, $(x,y)\cdot \nu=0$ and $\eta(t,x,y)\cdot \nu=0$. Therefore, by \eqref{boundarytorsion}, which writes the torsional rigidity in the expression of a boundary integral, we have
		\begin{align*}
			\frac{d}{dt}T(\Omega_t)=&-\frac{1}{2}(1+t)^{-3}\int_{B_tC}|\nabla u(t)(x,y)|^2 (x,y)\cdot (\sin\beta_t,\cos\beta_t)\, ds\\
			&+(1+t)^{-3}\int_{B_tC}|\nabla u(t)(x,y)|^2(x-y\cot\alpha,0)\cdot (\sin\beta_t,\cos\beta_t)\, ds\\
			=&\frac{1}{2}(1+t)^{-3}\int_{B_tC}|\nabla u(t)(x,y)|^2\sin\beta_t \left(x-y\left(2\frac{\cos\alpha}{\sin \alpha}+\frac{\cos\beta_t}{\sin \beta_t}\right)\right)\, ds.
		\end{align*}
		Let $M_t$ be the midpoint of $B_tC$ and $\theta_t=\angle M_tAB_t \in (0,\pi/2)$. 
		By the law of sine,
		\begin{align*}
			|AB_t|=\frac{\sin(\alpha+\beta_t)}{\sin \beta_t}, \quad |B_tM_t|=\frac{1}{2}|B_tC|=\frac{\sin \alpha}{2\sin\beta_t}.
		\end{align*}
		Again by the law of sine,
		\begin{align*}
			\frac{\sin(\theta_t+\beta_t)}{|AB_t|}=\frac{\sin\theta_t}{|B_tM_t|}.
		\end{align*}
	Hence
		\[
		\frac{\sin(\alpha+\beta_t)}{\sin(\theta_t+\beta_t)} 
		= \frac{\sin\alpha}{2\sin\theta_t}.
		\]
That is,
		\[
		2(\sin\alpha\cos\beta_t+\cos\alpha\sin\beta_t)\sin\theta_t
		= \sin\alpha(\sin\theta_t\cos\beta_t+\cos\theta_t\sin\beta_t).
		\]
After simplification, we have
		\[
		\cot\theta_t = 2\frac{\cos\alpha}{\sin\alpha} + \frac{\cos\beta_t}{\sin\beta_t}.
		\]
		
		Hence
		\begin{align*}
			\frac{d}{dt}T(\Omega_t)=\frac{1}{2}(1+t)^{-3}\sin \beta_t\cot\theta_t \int_{B_tC}|\nabla u(t)(x,y)|^2\left(x\tan\theta_t-y\right)\, ds.
		\end{align*}
		Note that for any $p_t=(x,y) \in B_tM_t$, we let $p_t'=(x',y')\in M_tC$ such that $|p_t'M_t|=|p_tM_t|$. When $t>0$, $|AB_t|>|AC|$, and thus by symmetry and Theorem \ref{symmetrycomparison},
		we have
		\begin{align*}
			0<x\tan\theta_t-y=-(x'\tan\theta_t-y'),\quad |\nabla u(p_t)|<|\nabla u(p_t')|.
		\end{align*}
		Therefore, for $t>0$, we have
		\begin{align*}
			\frac{d}{dt}T(\Omega_t)=\frac{1}{2}(1+t)^{-3}\sin \beta_t\cot\theta_t \int_{B_tM_t}\left(|\nabla u(t)(p_t)|^2-|\nabla u(p_t')|^2\right)\left(x\tan\theta-y\right)\, ds<0.
		\end{align*}
		
		Next, we prove the monotonicity of $\lambda_1(\Omega_t)$. Let $v(t)$ be the first eigenfunction of the Dirichlet Laplacian on $D_t$, with $\int_{D_t} v(t)^2\, dx=1$. Using the scaling relation
		\begin{align*}
			\lambda_1(\alpha \Omega)=\alpha^{-2}\lambda_1(\Omega),
		\end{align*}
		and by \eqref{boundarylambda1}, we have
		\begin{align*}
			\frac{d}{dt}\lambda_1(\Omega_t)=& \frac{d}{dt}\left((1+t)\lambda_1(D_t)\right)\\
			=& \lambda_1(D_t)-(1+t)\int_{\partial D_t}|\nabla v(t)(x,y)|^2\eta(t,x,y)\cdot \nu\, ds\\
			=& \frac{1}{2}\int_{\partial D_t}|\nabla v(t)(x,y)|^2 (x,y)\cdot \nu\, ds-\int_{\partial D_t}|\nabla v(t)(x,y)|^2(x-y\cot\alpha,0)\cdot \nu\, ds\\
			=&\frac{1}{2}\int_{B_tC}|\nabla v(t)(x,y)|^2\sin\beta_t \left(y\left(2\frac{\cos\alpha}{\sin \alpha}+\frac{\cos\beta_t}{\sin \beta_t}\right)-x\right)\, ds.
		\end{align*}
		By the similar argument as before, we derive $\tfrac{d}{dt}\lambda_1(F_t(\Omega))>0$.
	\end{proof}
	
	The sign analysis in the proof above is illustrated in Figure \ref{fig:add2} below.

	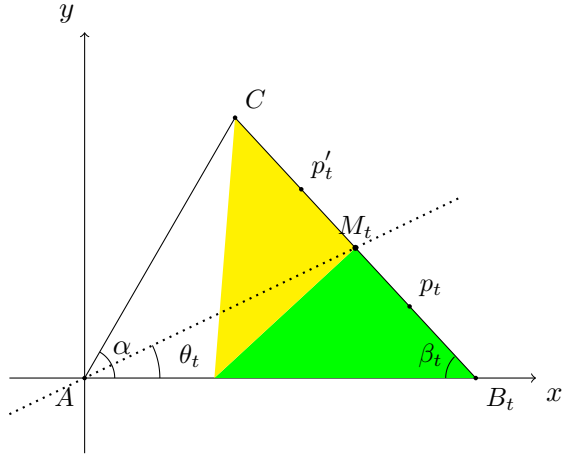
\begin{figure}[htp]
		\centering
		\begin{tikzpicture}[scale = 1.5]
			\pgfmathsetmacro\d{180/pi};
			\pgfmathsetmacro\S{sin(60)};
			\pgfmathsetmacro\C{cos(60)};
			\pgfmathsetmacro\m{\C+1.3};
			\pgfmathsetmacro\mx{\C+1.3};
			\pgfmathsetmacro\my{\S};
			\pgfmathsetmacro\re{\my*(\S/(\C-1.3))+\mx};
			\pgfmathsetmacro\px{1.2*\m};
			\pgfmathsetmacro\pxx{0.8*\m};
			\pgfmathsetmacro\py{\S/(\C-1.3)*(\px-2.6)};
			\pgfmathsetmacro\pyy{\S/(\C-1.3)*(\pxx-2.6)};
			
			\draw[->] (-0.5,0)--(3,0) node[below right] {$x$};
			\draw[->] (0,-0.5)--(0,2.3) node[above left] {$y$};
			
			\fill[green,opacity=0.4] (\re,0)--(2.6,0)--(\mx,\my)--cycle;
			\fill[yellow,opacity=0.4] (\re,0)--(2*\C,2*\S)--(\mx,\my)--cycle;
			
			\draw (0,0)--(2*\C,2*\S);
			\draw (2*\C,2*\S)--(2.6,0);
			\fill (0,0) circle (0.015) node[below left] {\small $A$};
			\fill (2*\C,2*\S) circle (0.015) node[above right] {\small $C$};
			\fill (2.6,0) circle (0.015) node[below right] {\small $B_{t}$};
			
			\draw[samples=500,black, domain=0:pi/3, variable=\t] plot ({0.2*cos(\t*\d)},{0.2*sin(\t*\d)});
			\node at (0.25,0.2) { $\alpha$};
			
			\draw[samples=500,black, domain=3*pi/4:pi, variable=\t] plot ({2.6+0.2*cos(\t*\d)},{0.2*sin(\t*\d)});
			\node at (2.3,0.15) {\small $\beta_{t}$};
			
			\draw[samples=500,thick,dotted, domain=-0.5:2.5, variable=\t] plot ({\t},{\S/(1.3+\C)*\t});
			
			\fill (1.3+\C,\S) circle (0.02) node[above] {\small $M_{t}$};
			
			\draw[samples=500,black, domain=0:pi/7, variable=\t] plot ({0.5*cos(\t*\d)},{0.5*sin(\t*\d)});
			\node at (0.7,0.15) {\small  $\theta_{t}$};
			
			\fill (\pxx,\pyy) circle (0.015) node[above right] {\small $p'_{t}$};
			\fill (\px,\py) circle (0.015) node[above right] {\small $p_{t}$};

		\end{tikzpicture}
		\caption{For any $t>0$, $|AB_t|>|AC|$. $M_t$ is the midpoint of 
			$B_tC$, $\theta_t=\angle B_tAM_t$. If $|p_t M_t|=|p_t'M_t|$, then $|\nabla u(t)(p_t)|<|\nabla u(t)(p_t')|$.}
		\label{fig:add2}
	\end{figure}

	Next, we shall apply Theorem \ref{yangsheng1} to prove Corollaries \ref{yangshengcor1}-\ref{yangshengcor2}.
	
	\begin{proof}[Proof of Corollary \ref{yangshengcor1}]
	It suffices to prove the monotonicity of $\widetilde{T}(\cdot):=T(\cdot)/|\cdot|^2$ along the inward angle bisector flow, and the case of $\lambda_1(\cdot)$ is similar.
	
	Fix $0\le s<t\le t_0$.  
	Set $\triangle_s:=\triangle_{BA_sC}$ and $\triangle_t:=\triangle_{BA_tC}$.  
	We must show $\widetilde{T}(\triangle_s)<\widetilde{T}(\triangle_t)$.

	Draw the line $CA_t$ and let it meet the line segment $BA_s$ at a point $D$. Clearly,
	\begin{align}
			\label{hangzhou1}
			|BD|<|BA_s|,\qquad |CD|>|CA_t|.
	\end{align}

	In $\triangle BDC$, observe that $\angle BDC<\angle BA_tC$.  Since $t<t_0$, in $\triangle BA_tC$, we have $|BC|<|A_tB|$, and thus $\angle BA_tC<\angle BCA_t$. In view that $\angle BCA_t=\angle BCD$ because $D$ is on the ray $CA_t$, we have 
	\[
	\angle BDC<\angle BA_tC<\angle BCD .
	\]
	
		Therefore,
	\begin{align}
		\label{hangzhou2}
		|BC|<|BD|.
	\end{align}

	Triangles $\triangle_s$ and $\triangle_{BDC}$ share the angle at $B$.  
	For this fixed angle, by the first inequality in \eqref{hangzhou1} and \eqref{hangzhou2},
	\[
	\frac{|BA_s|}{|BC|}>\frac{|BD|}{|BC|}>1.
	\]
Therefore, Theorem \ref{yangsheng1} yields
	\begin{align}
		\label{hangzhou3}
			\widetilde{T}(\triangle BDC)>\widetilde{T}(\triangle_s).
	\end{align}
Triangles $\triangle BDC$ and $\triangle_t$ share the angle at $C$.  
	For this fixed angle, since $|A_tC|>|BC|$ and the second inequality in \eqref{hangzhou1}, we have
	\[
	\frac{|CD|}{|BC|}>\frac{|CA_t|}{|BC|}>1.
	\]
	Applying Theorem \ref{yangsheng1} again gives
\begin{align}
	\label{hangzhou4}
	\widetilde{T}(\triangle_{ BDC})<\widetilde{T}(\triangle_t).
\end{align}
Combining \eqref{hangzhou3} and \eqref{hangzhou4}, we obtain
	\[
	\widetilde{T}(\triangle_s)<\widetilde{T}(\triangle_{BDC})<\widetilde{T}(\triangle_t),
	\]
	which is exactly $\widetilde{T}(\triangle_s)<\widetilde{T}(\triangle_t)$.  
	Since $s<t$ is arbitrary, the function $t\mapsto\widetilde{T}(\triangle_{BA_tC})$ is strictly increasing on $[0,t_0]$.
	\end{proof}
	
The proof of Corollary \ref{yangsheng2'} will be similar to the previous argument, and we only give a sketch below.
\begin{proof}[Proof of Corollary \ref{yangshengcor2}]
	When $0<s<t<t_1$, we extend the line segment $BC_s$ and let it meet the line segment $AC_t$ at a point $D$. By the assumptions, we have
	\[1< \frac{|AB|}{|AC_t|}< \frac{|AB|}{|AD|}.\]
	By Theorem \ref{yangsheng1}, we have
	\begin{align}
		\label{hangzhou5}
	\widetilde{T}(\triangle_{ABC_t})>	\widetilde{T}(\triangle_{ABD}).
	\end{align}
	Also, since $|AB|>|BC_t|$, we have
	\[\angle ADB>\angle AC_tB>\angle BAC_t=\angle BAD,\]
	and thus $|AB|>|BD|$. Clearly $|BD|>|BC_s|$, and we therefore have
	\[1<\frac{|AB|}{|BD|}<\frac{|AB|}{|BC_s|}.\]
	Hence by Theorem \ref{yangsheng1},	\begin{align}
		\label{hangzhou6}
		\widetilde{T}(\triangle_{ABD})>	\widetilde{T}(\triangle_{ABC_s}).
	\end{align}
	Combining \eqref{hangzhou5} and \eqref{hangzhou6}, we have
	\begin{align*}
			\widetilde{T}(\triangle_{ABC_t})>\widetilde{T}(\triangle_{ABC_s}).
	\end{align*}
	This finishes the proof.
\end{proof}
	
Last, we remark that the proof of Corollary \ref{yangshengcor1} does not depend specifically on the motion of vertex $A$ along the angle bisector. The essential condition is that during the motion, the ordering of the side lengths remains unchanged. Consequently, the same argument yields the following stronger result:

\begin{corollary}
	\label{thm:general-motion}
	Let $\triangle ABC$ be a triangle satisfying $|AB| > |AC| > |BC|$.  
	Fix the side $BC$ and let $\{A_t\}_{t\in[0,T]}$ be a continuous family of points such that:
	\begin{enumerate}
		\item $A_0 = A$;
		\item For each $0 \le s < t \le T$, the point $A_t$ lies in the interior of $\triangle A_s B C$;
		\item For all $t \in [0,T]$, the side lengths satisfy
		\[
		|A_t B| > |A_t C| > |BC|.
		\tag{H}
		\]
	\end{enumerate}
	Then the normalized torsional rigidity
	\[
	t \longmapsto \frac{T(\triangle BA_tC)}{|\triangle BA_tC|^2}
	\]
	is strictly increasing on $[0,T]$.
\end{corollary}

Similarly, the following result can also be derived from the proof of Corollary \ref{yangshengcor2}:
\begin{corollary}
	\label{thm:general-motion}
	Let $\triangle ABC$ be a triangle satisfying $|AB| > |AC| \ge |BC|$.  
	Fix the side $BC$ and let $\{C_t\}_{t\in[0,T]}$ be a continuous family of points such that:
	\begin{enumerate}
		\item $C_0 = C$;
		\item For each $0 \le s < t \le T$, the point $C_s$ lies in the interior of $\triangle A B C_t$;
		\item For all $t \in [0,T]$, the side lengths satisfy
		\[
		|A B| > |AC_t | \ge  |BC_t|.
		\tag{H}
		\]
	\end{enumerate}
	Then the normalized torsional rigidity
	\[
	t \longmapsto \frac{T(\triangle ABC_t)}{|\triangle ABC_t|^2}
	\]
	is strictly increasing on $[0,T]$.
\end{corollary}
These two corollaries also imply Theorem \ref{yangsheng2} and Theorem \ref{yangsheng2'}.

	\section{A New proof of the Saint-Venant inequality via the mean curvature flow}
Throughout this and the next section, we let $B_1$ denote the unit ball in $\mathbb{R}^n$, and a point in $\mathbb{R}^n$ is denoted by a single letter $x$.

 The Saint-Venant inequality states that for any bounded domain $\Omega \subset \mathbb{R}^n$, we have
	\begin{align*}
		\frac{T(\Omega)}{|\Omega|^{\frac{n+2}{n}}}\le  \frac{T(B_1)}{|B_1|^{\frac{n+2}{n}}},
	\end{align*}
	where the equality holds if and only if $\Omega$ is a ball up to a set of zero capacity. 
	
	In this section, we use the mean curvature flow to give a new proof of this inequality in the smooth convex setting. First, we state the following lemma, which is motivated from the argument used in \cite{Reilly}.
	
	\begin{lemma}
		\label{tminequality}
		Let $\Omega$ be a bounded smooth domain in $\mathbb{R}^n$, $H$ be the mean curvature on $\partial \Omega$ and $u$ be the torsion function on $\Omega$. Then,
		\begin{align*}
			\int_{\partial \Omega}|\nabla u|^2 H\, d\sigma\le \frac{n-1}{n}|\Omega|,
		\end{align*}
		with the equality holding if and only if $\Omega$ is a ball.
	\end{lemma}
	
	\begin{proof}
		On the one hand, by the torsion equation and the Schwarz inequality, we obtain
		\begin{align}
			\label{1.1}
			\int_{\Omega}\Delta \left(\frac{1}{2}|\nabla u|^2\right)\, dx=&\int_{\Omega} |\nabla ^2 u|^2 \, dx+ \int_{\Omega} \nabla(\Delta u)\cdot \nabla u\, dx\nonumber\\
			 \ge &\int_\Omega \frac{(\Delta u)^2}{n}\, dx = \frac{1}{n}|\Omega|.
		\end{align}
		On the other hand, we have
		\begin{align}
			\label{jieshi}
			\int_{\Omega}\Delta \left(\frac{1}{2}|\nabla u|^2\right)\, dx
			=& \int_{\partial \Omega} \nabla^2u \nabla u \nu \, d\sigma\nonumber\\
			=& \int_{\partial \Omega} u_{\nu} u_{\nu\nu}\, d\sigma\quad \mbox{since $\nabla u \perp \partial \Omega$} \nonumber\\
			=& \int_{\partial \Omega} u_{\nu}(-1-Hu_{\nu})\, d\sigma\nonumber\\
			=& |\Omega|-\int_{\partial \Omega}Hu_{\nu}^2\, d\sigma.
		\end{align}
		In the above, we have used the torsion equation and the well-known formula that on $\partial \Omega$,
		\begin{align}
			\Delta u=  u_{\nu\nu}+H\frac{\partial u}{\partial \nu}+\Delta_{\partial \Omega}u.
		\end{align}
		Hence by \eqref{1.1} and \eqref{jieshi}, we have
		\begin{align*}
			\int_{\partial \Omega}|\nabla u|^2 H\, d\sigma=\int_{\partial \Omega}u_\nu^2 H\, d\sigma\le \frac{n-1}{n}|\Omega|.
		\end{align*}
		The equality holds if and only if in $\Omega$, $\nabla^2 u$ has $n$ same eigenvalues, all equal to $1/n$. That is, in $\Omega$, $\nabla^2 u=\tfrac{1}{n}I_n$ where $I_n$ is the identity matrix. Hence via integration, $u$ must be a quadratic function $\tfrac{1}{2n}|x-x_0|^2+b$ with $b\in \mathbb{R}$. Threfore, by the Dirichlet condition of $u$, $\Omega$ must be a ball.
	\end{proof}

	Now we are ready to prove Theorem \ref{svinequalitynintro}.
	
	\begin{proof}[Proof of Theorem \ref{svinequalitynintro}]
		Let $F_t$ be the mean curvature flow, that is,
		\begin{align*}
			\partial_t F_t(x)=-H\nu_{F_t(x)},\quad F_0(x)=x,
		\end{align*}
		where $\nu_{F_t(x)}$ is the unit outer normal to the hypersurface $F_t(\partial \Omega)$. Let $\Omega_t$ be the domain enclosed by $F_t(\partial \Omega)$.
		
	Due to Huisken's seminal work \cite{Huisken84}, 
	the mean curvature flow starting from a smooth closed strictly convex hypersurface $\partial\Omega$ 
	contracts it to a point in finite time $t_0$ while preserving convexity. 
	To describe the asymptotic shape, one rescales the flow by a time-dependent factor so that the enclosed volume remains constant. 
	More precisely, if we set $\tilde\Omega_t := \lambda(t) \Omega_t$ with $\lambda(t)>0$ chosen so that $|\tilde\Omega_t| = |\Omega|$, 
	then $\partial\tilde\Omega_t$ converges in $C^\infty$ to $\partial B_1$ as $t\to t_0^-$.

		By scaling invariance and a result of Chenais (see \cite[Section 2.3]{Henrot} which applies also to the torsional rigidity), we thus have
		\begin{align}
			\label{limitvalue}
			\lim_{t\rightarrow t_0^-}
			\frac{T(\Omega_t)}{|\Omega_t|^{\frac{n+2}{n}}}=\frac{T(B_1)}{|B_1|^{\frac{n+2}{n}}}= \frac{\omega_n^{-\frac{2}{n}}}{n(n+2)},
		\end{align}
		where $\omega_n$ is the volume of the unit ball in $\mathbb{R}^n$.
		
		Let $g(t)=T(\Omega_t)-C_n|\Omega_t|^{(n+2)/n}$, where $C_n$ is the constant in the right-hand side of \eqref{limitvalue}. By Theorem \ref{Tshapederivative}, for any $0\le t<t_0$, we have
		\begin{align}
			\label{3dg'}
			g'(t)=-\left(\int_{\partial \Omega_t}|\nabla u(t)|^2 H\, d\sigma\right) +C_n\frac{n+2}{n}|\Omega_t|^{\frac{2}{n}}\int_{\partial \Omega_t}H\, d\sigma,
		\end{align}
		where in the above we have used the formula
		\begin{align*}
			\frac{d}{dt}|\Omega_t|=-\int_{\partial \Omega_t}H\, d\sigma.
		\end{align*}
		This formula can also be derived directly from Theorem \ref{Tshapederivative} by choosing the flow with velocity $\eta=-H\nu$ on the boundary. By the classical Minkowski inequality for convex bodies (see for example \cite{Sch}) and the isoperimetric inequality, we have
		\begin{align}
			\label{lowerbound}
			\int_{\partial \Omega_t}H\, d\sigma\ge & (n-1)|\mathbb{S}^{n-1}|^{\frac{1}{n-1}}|\partial \Omega_t|^{\frac{n-2}{n-1}}\nonumber\\
			\ge & (n-1)(n\omega_n)^{\frac{1}{n-1}}\left(n\omega_n^{\frac{1}{n}}|\Omega_t|^{\frac{n-1}{n}}\right)^{\frac{n-2}{n-1}}\nonumber\\
			=& n(n-1)\omega_n^{\frac{2}{n}}|\Omega_t|^{\frac{n-2}{n}}.
		\end{align}
		Therefore, by \eqref{limitvalue}, \eqref{3dg'}, \eqref{lowerbound} and Lemma \ref{tminequality}, 
		the following differential inequality holds for any $0\le t<t_0$:
		\begin{align}
			\label{g'inequality}
			g'(t)\ge& -\frac{n-1}{n}|\Omega_t|+C_n(n+2)(n-1)\omega_n^{\frac{2}{n}}|\Omega_t|=0.
		\end{align}
	\end{proof}
	
	As a consequence, we have
	
	\begin{corollary}
		\label{svinequalityn}
		Let $\Omega$ be a smooth and convex domain in $\mathbb{R}^n$. Then,
		\begin{align*}
			\frac{T(\Omega)}{|\Omega|^{\frac{n+2}{n}}}\le \frac{T(B_1)}{|B_1|^{\frac{n+2}{n}}},
		\end{align*}where the equality holds if and only if $\Omega$ is a ball. 
	\end{corollary}
	
	\begin{proof}
		We use the notations in the proof of Theorem \ref{svinequalitynintro}. By the theorem, $g(t)$ is an increasing function for $t\in [0,t_0)$. Since $g(t_0^-)=0$, we have in particular $g(0)\le 0$, and this implies the Saint-Venant inequality
		\begin{align*}
			\frac{T(\Omega)}{|\Omega|^{\frac{n+2}{n}}}\le \frac{T(B_1)}{|B_1|^{\frac{n+2}{n}}}.
		\end{align*}
		
		If the equality case holds, then $g(t)\equiv 0$ for $t\in [0,t_0)$ and thus in particular $g'(0)=0$. From the proof of Theorem \ref{svinequalitynintro}, this entails that the equality case holds in Lemma \ref{tminequality}. Therefore, $\Omega$ must be a ball.
		
	\end{proof}
	
	When $n=2$, the mean curvature flow becomes the curve shortening flow. Then, the convexity assumption can be relaxed to a simply connected assumption on the initial domain, to guarantee that there are no singularities during the evolution, see \cite{Grayson}. Also, the total mean curvature along $\partial \Omega_t$ is always $2\pi$, and thus the proof of Theorem \ref{svinequalitynintro} still works. Therefore, when $n=2$, Theorem \ref{svinequalitynintro} also holds for smooth simply connected bounded domains.
	
	\begin{remark}
		While recent advances in the analysis of singularities in mean curvature flow might eventually allow the convexity hypothesis in Theorem \ref{svinequalitynintro} to be relaxed, such an investigation lies beyond the scope of the present paper. Instead, we focus on a more pressing open question that arises directly from our work: whether the normalized torsional rigidity $T(\cdot)/|\cdot|^{(n+2)/n}$ is monotone increasing along the mean curvature flow, even for initially smooth convex planar domains. It is important to note that this property is not equivalent to the weak monotonicity established in \eqref{mcfweak}; its analysis is taken up in Section 6, where we examine a related shape functional $Q(\cdot)$, as stated in \eqref{qfunctional}.
				
		The subtlety of this question is underscored by the fact that analogous monotonicity properties for other classical functionals often depend sensitively on the choice of flow and the dimension. A paradigmatic case is the isoperimetric ratio $P(\cdot)/|\cdot|^{(n-1)/n}$. It is known to decrease along \emph{normalized} versions of both the mean curvature flow and the inverse mean curvature flow in all dimensions \cite{GL15, GL21}. However, its behavior under the \emph{standard (unnormalized)} flows differs: monotonicity holds for the curve shortening flow in two dimensions ($n=2$) \cite{Gage} but fails in higher dimensions \cite{M21}. (In this context, we note that Proposition \ref{IMCFproof} establishes the monotonicity of the isoperimetric ratio under the \emph{unnormalized} inverse mean curvature flow in any dimension—a result that, to our knowledge, has not been explicitly stated elsewhere.) In stark contrast, the monotonicity of the normalized torsional rigidity $T(\cdot)/|\cdot|^{(n+2)/n}$ under \emph{any} of these flows (normalized or unnormalized) remains a completely open problem, even in the planar case.
	\end{remark}
	
	At the end of the section, we also mention that it is possible to find another flow which involves the torsion function, to guarantee that $T(\cdot)/|\cdot|^{(n+2)/n}$ is monotone, yet the cost requires us to prove the convergence of convex domains to a ball up to scaling, along the flow. This would be another interesting topic and deserves further research. See section 9 below.  
	
\section{Extremal properties of the functional $Q(\cdot)$ given by \eqref{qfunctional}}

Let $\Omega$ be a smooth convex domain in $\mathbb{R}^n$, and let $u(t)$ be the torsion function over $\Omega_t$, where  $\Omega_t$ is the family of domains evolving by the mean curvature flow with initial data $\Omega_0=\Omega$. Then by direct computation, that \begin{align*}
	\frac{d}{dt}\left(\frac{T(\Omega_t)}{|\Omega_t|^{(n+2)/n}}\right)\ge 0 \quad \mbox{for all $0<t<T$},
\end{align*} is equivalent to 
\begin{align*}
	Q(\Omega_t)\le \frac{n+2}{n}=Q(B_1),\, \mbox{for all $0<t<T$},
\end{align*}
where $Q(\cdot)$ is defined in \eqref{qfunctional}. This motivates us to study the functional $Q(\Omega)$, which is scaling invariant in the sense that $Q(t\Omega)=Q(\Omega)$ for any $t>0$, and we expect and conjecture that among all bounded convex domains, $Q(\cdot)$ achieves its maximum only at round domains at least in two dimensions. 

To prove the global maximality of balls to $Q(\cdot)$ is rather challenging as explained in the introduction. So far, we can prove some partial extremal results related to this direction. First, we have:

\begin{proposition}
	For any bounded convex smooth domain $\Omega$ in $\mathbb{R}^n$, $Q(\Omega)$ is uniformly bounded from the above.
\end{proposition}

\begin{proof}
	Let $u$ be the torsion function over $\Omega$. Then, we have
	$$Q(\Omega)=\frac{|\Omega|\int_{\partial \Omega}|\nabla u|^{2}H d\sigma}{T(\Omega)\int_{\partial \Omega}H\, d\sigma}\leq \frac{|\Omega|\|\nabla u\|^{2}_{L^{\infty}(\partial \Omega)}}{T(\Omega)}< \frac{2|\Omega|\|u\|_{L^{\infty}(\Omega)}}{T(\Omega)},$$
	where we used the non-negativity of $H$ since $\Omega$ is convex and the following classical estimate via $P$-function (see \cite[Eq. 6.12]{Sperb})
	\begin{align}
		\label{unigradientbound}
		|\nabla u_{\Omega}|^{2}< 2\|u\|_{L^{\infty}(\Omega)}.
	\end{align}
	The upper bound (though not sharp in general)
	\begin{align}
		\label{henrotbound}
	\frac{|\Omega|\|u\|_{L^{\infty}(\Omega)}}{T(\Omega)}\leq (n+1)^2	
	\end{align}
	was given by \cite[Theorem 2.6]{TLM1} by exploring the concavity of $u^{1/2}$. Summarizing above inequalities, we have the upper bound
	$$Q(\Omega)< 2(n+1)^2.$$
\end{proof}

Let $u=u_\Omega$ be the torsion function in $\Omega$. Observe that by \eqref{jieshi}, we have 
\begin{align*}
	\int_{\partial \Omega}|\nabla u|^{2}H d\sigma=|\Omega|-\int_{\Omega} |\nabla^2 u|^2\, dx.
\end{align*}
Hence when $n=2$, and in view that $-\Delta u=1$, we have
\begin{align*}
		\int_{\partial \Omega}|\nabla u|^{2}H d\sigma=|\Omega|-\int_{\Omega} \left((\Delta u)^2-2det(D^2 u)\right)\, dx=2\int_{\Omega}det(D^2 u)\, dx.
	\end{align*}
Also, since $\int_{\partial \Omega}H\, d\sigma=2\pi$ when $n=2$, we have
\begin{align}
	\label{revisedQfunctional}
	Q(\Omega)=\frac{|\Omega|\int_{\Omega} det(D^2 u)\, dx}{\pi T(\Omega)}.
\end{align}
From this expression, $Q(\Omega)$ can be explicitly computed when $\Omega$ is enclosed by an ellipse, and thus the following proposition is obtained.

\begin{proposition}
	Among bounded domains enclosed by ellipses, balls uniquely maximize $Q(\cdot)$.
\end{proposition}

\begin{proof}
Since the functional $Q(\cdot)$ is scaling invariant, we just need to consider the elliptic planar domains with area $\pi$. Let $\Omega_{\lambda}\, (0<\lambda\le 1)$ be the family of such domains defined as $\Omega_{\lambda}:=\{(x,y)\in \R^{2}: \lambda^{-2}x^{2}+\lambda^{2}y^{2}<1\}.$
	One can check that the torsion function in $\Omega_{\lambda}$ is 
	\begin{align*}
		u_{\Omega_{\lambda}}(x,y)=\frac{1-\lambda^{-2}x^{2}-\lambda^{2}y^{2}}{2(\lambda^{-2}+\lambda^{2})}.
	\end{align*}
	Directly calculation yields
	\begin{align*}
		\det \Big(D^{2}u_{\Omega_{\lambda}}(x,y)\Big)=(\lambda^{-2}+\lambda^{2})^{-2}.
	\end{align*}
	Integrating $u_{\Omega_{\lambda}}$ and $\det(D^{2}u_{\Omega_{\lambda}})$ in $\Omega_{\lambda}$, we have
	\begin{align*}
		&T(\Omega_{\lambda})=\int_{\Omega_{\lambda}}u_{\Omega_{\lambda}}\,dx=\frac{\pi}{4(\lambda^{-2}+\lambda^{2})},\\
		&\int_{\Omega_\lambda}\det(D^{2}u_{\Omega_\lambda})\,dx=\frac{\pi}{(\lambda^{-2}+\lambda^{2})^{2}}.
	\end{align*}
	Therefore, we have
	$$Q(\Omega_{\lambda})=\frac{|\Omega_{\lambda}|\int_{\Omega_\lambda}\det(D^{2}u_{\Omega_\lambda})\,dx}{\pi T(\Omega_{\lambda})}=\frac{4}{\lambda^{2}+\lambda^{-2}}.$$
	Thus $Q(\Omega_{\lambda})$ is strictly increasing to $Q(\Omega_{1})=Q(B_1)$ as $\lambda$ goes from $0$ to $1$.
\end{proof}
	
To study the local optimality of shapes, similar to the framework of \cite{HLL22}, we adopt the method of perturbing domains along flow maps, which can simplify the computation of shape derivatives. We first recall some terminologies. 

Let \(\eta \in C_0^\infty(\R^n, \R^n)\) be a smooth vector field. We define the associated flow map \(F_t(x) = F(t,x)\) as the solution to
\[
\begin{cases}
	\frac{d}{dt} F(t,x) = \eta(F(t,x)), & t \in (-\varepsilon, \varepsilon), \\[8pt]
	F(0,x) = x.
\end{cases}
\]
Then \(F_t\) is a local diffeomorphism for \(|t|\) sufficiently small. We say that \(F_t\) (or \(\eta\)) preserves the volume of a domain \(\Omega\) if \(|F_t(\Omega)| = |\Omega|\) for all small \(t\). We say that $F_t$ is \textit{a local translation map} of $\Omega$ if there exists a constant vector $V\in \mathbb{R}^n$ such that on $\partial \Omega$, $\eta \cdot \nu=V\cdot \nu$, where $\nu$ is the unit outer normal to $\partial \Omega$.

Let \(\mathcal{E}\) be a \(C^2\) shape functional. We say that \(\Omega\) is \textit{stationary} for \(\mathcal{E}\) under volume constraint if
\[
\left.\frac{d}{d t} \right|_{t=0} \mathcal{E}(F_t(\Omega)) = 0
\]
for every smooth flow map \(F_t\) that preserves the volume of \(\Omega\).  

We say that \(\Omega\) is a \textit{strict local maximizer} for \(\mathcal{E}\) under volume constraint if \(\Omega\) is stationary and, for every volume‑preserving flow map \(F_t\) that is not a local translation map of $\Omega$,
\[
\left.\frac{d^2}{d t^2}\right|_{t=0} \mathcal{E}(F_t(\Omega)) < 0.
\]

With these notions we can state the following local optimality result.

\begin{theorem}
	\label{thm:main}
	Let $B_1 \subset \mathbb{R}^2$ be the unit disk of centered at the origin. Then $B_1$ is a strict local maximizer to $Q(\cdot)$.
\end{theorem}

In order to determine the local maximality of ball to $Q(\cdot)$, we first need the following first and second variation formula of the functional \begin{align}
	\label{dfunctional}
D(\Omega):=\int_{\Omega} det(D^2 u_\Omega)\, dx.	
\end{align}
	
\begin{proposition}
	\label{pro:DM}
	Let $B_{1}\subset \R^{2}$ be the unit disk and $F_{t}(\cdot)$ be the map generated by a smooth time-dependent volume-preserving vector field $\eta$ with compact support. We denote $\Omega_{t}:=F_{t}(B_{1})$. Then we have
	\begin{align}
		\label{eq:d1}
		\frac{d}{dt}\Big|_{t=0}D(\Omega_{t})=0,
	\end{align}
	and 
	\begin{align}
		\label{eq:d2}
		\frac{d^{2}}{dt^{2}}\Big|_{t=0}D(\Omega_{t})=2\int_{B_{1}}\det(D^{2}h)dx, 
	\end{align}
	where $h$ is the harmonic function in $B_{1}$ satisfies the boundary condition
	$$h+\nabla u_{B_{1}}\cdot \eta=0, \quad \mbox{on $\partial B_{1}$}.$$
\end{proposition}

	For abbreviation, we simply denote the torsion function in $\Omega_{t}$  by $u^{t}$ instead of $u(t)$ as in the proof of Theorem \ref{Tshapederivative}, and throughout this section, we adopt the following convention: 
	a dot denotes the shape derivative with respect to the flow parameter $t$. That is,
	$$\dot u^t = u'(t),\qquad \ddot u^t = u''(t),$$
	where $u'(t)$ is as in the proof of Theorem \ref{Tshapederivative} and $u''(t)$ is defined in a similar manner. 

\begin{proof}
	Let $\sigma_{t}$ be the volume element on $\partial \Omega_{t}$. By applying the Hadamard formula, we have
	\begin{align*}
		\frac{d}{dt}D(\Omega_{t})&=\int_{\Omega_{t}}\dot{u}^{t}_{11}u^{t}_{22}+u^{t}_{11}\dot{u}^{t}_{22}-2u^{t}_{12}\dot{u}^{t}_{12}\,dx+\int_{\partial \Omega_{t}}\det(D^{2}u^{t})\eta\cdot \nu \,d\sigma_{t}\\
		&=\int_{\Omega_{t}}\dot{u}^{t}_{11}u^{t}_{22}+u^{t}_{11}\dot{u}^{t}_{22}-2u^{t}_{12}\dot{u}^{t}_{12}+\det(D^{2}u^{t})\div \eta+\eta \cdot\nabla \det(D^{2}u^{t})\, dx,
	\end{align*}
	where we used the divergence theorem. Applying the Hadamard formula again, we have
	\begin{align*}
		\frac{d^{2}}{d t^{2}}D(\Omega_{t})&=\int_{\Omega_{t}}\ddot{u}^{t}_{11}u^{t}_{22}+2\dot{u}^{t}_{11}\dot{u}^{t}_{22}+u^{t}_{11}\ddot{u}^{t}_{22}-2(\dot{u}^{t}_{12})^{2}-2u_{12}^{t}\ddot{u}_{12}^{t}\,dx\\
		&+\int_{\Omega_{t}}\Big(\dot{u}^{t}_{22}\nabla u_{11}^{t}+u^{t}_{22}\nabla \dot{u}_{11}^{t}+\dot{u}^{t}_{11}\nabla u^{t}_{22}+u^{t}_{11}\nabla \dot{u}^{t}_{22}-2\dot{u}^{t}_{12}\nabla u^{t}_{12}-2u^{t}_{12}\nabla \dot{u}^{t}_{12}\Big)\cdot \eta\, dx\\
		&+\int_{\Omega_{t}}\Big(\dot{u}^{t}_{11}u^{t}_{22}+u^{t}_{11}\dot{u}^{t}_{22}-2u^{t}_{12}\dot{u}^{t}_{12}\Big)\div \eta\, dx+\int_{\partial \Omega_{t}}\Big(\dot{u}^{t}_{11}u^{t}_{22}+u^{t}_{11}\dot{u}^{t}_{22}-2u^{t}_{12}\dot{u}^{t}_{12}\Big) \eta\cdot \nu\, d\sigma\\
		&+\int_{\partial \Omega_{t}}\Big(\det(D^{2}u^{t}) \div \eta +\eta \cdot \nabla \det(D^{2}u^{t})\Big)\eta\cdot \nu\, d\sigma.
	\end{align*}
	When $t=0$, $u^{0}(x)=\frac{1}{4}(1-|x|^{2})$, and thus 
	\begin{align*}
		u^{0}_{11}=u^{0}_{22}=-\frac{1}{2},\quad \mbox{and}\quad u^{0}_{12}=0 \qquad \mbox{in $B_{1}$},
	\end{align*}
	Since $\dot{u}$ and $\ddot{u}$ are harmonic functions, we have
	\[\dot{u}^{0}_{11}u^{0}_{22}+u^{0}_{11}\dot{u}^{0}_{22}=-\frac{1}{2}\Delta \dot{u}^{0}=0,\]
	$$\ddot{u}_{11}^{0}u_{22}^{0}+u^{0}_{11}\ddot{u}^{0}_{22}=-\frac{1}{2}(\Delta \ddot{u}^{0})=0,$$
	and 
	$$\Big(u_{22}^{0}\nabla \dot{u}^{0}_{11}+u^{0}_{11}\nabla \dot{u}_{22}^{0}\Big)\cdot \eta=-\frac{1}{2}\nabla (\Delta \dot{u}^{0})\cdot \eta=0.$$
Also, since $\eta$ is volume-preserving, it satisfies the following two equalities
\begin{align}
	\label{volumepreserving}
\int_{\partial \Omega_{t}}\eta\cdot \nu_{\partial \Omega_{t}}d\sigma_{t}=0 \quad\mbox{and}\quad \int_{\partial \Omega_{t}}\div(\eta)\eta\cdot \nu d\sigma_{t}=0.	
\end{align}

	Taking these identities into the first and second derivative formulas of $D(\Omega_t)$ at $t=0$, we obtain
	\begin{align}
		\label{eq:d1}
		\frac{d}{dt}\Big|_{t=0}D(\Omega_{t})=0,
	\end{align}
	and 
	\begin{align}
		\label{eq:d2}
		\frac{d^{2}}{dt^{2}}\Big|_{t=0}D(\Omega_{t})=2\int_{B_{1}}\det(D^{2}\dot{u}^{0})dx. 
	\end{align}
	Since $u^t=0$ on $\partial \Omega_t$, $\dot{u}^0$ satisfies
	$$\dot{u}^{0}+\nabla u^{0}\cdot \eta=0, \quad \mbox{on $\partial B_{1}$}.$$
	Hence the harmonic function $h$ in the proposition coincides with $\dot{u}^{0}$, and we complete the proof.
\end{proof}

With Proposition \ref{pro:DM}, we can now derive the first and second variation formula for $Q(\cdot)$ at a disk.
\begin{proposition}
	\label{pro:DQ}
	Let $F_{t}$ be a flow map generated by a smooth vector field $\eta$ which preserves the volume of the unit disk $B_{1}$ along the flow. Let $\Omega_{t}=F_{t}(B_{1})$. Then,
	\begin{align}
		\label{eq:d1Q}
		\frac{d}{dt}\Big|_{t=0}Q(\Omega_{t})=0
	\end{align}
	and 
	\begin{align}
		\label{eq:d2Q}
		\frac{d^{2}}{dt^{2}}\Big|_{t=0} Q(\Omega_{t})=\frac{16}{\pi} \Big(\int_{B_{1}}\det(D^{2}h)+2|\nabla h|^{2}dx-2\int_{\partial B_{1}}h^{2}d\sigma\Big),	\end{align}
	where $h$ is the harmonic function satisfying the boundary condition:
	$$h+\nabla u_{B_{1}}\cdot \eta=0, \quad \mbox{on $\partial B_{1}$}.$$
\end{proposition}
\begin{proof}
	By \eqref{revisedQfunctional}, directly calculations yields
	\begin{align*}
		\frac{d}{dt}\Big|_{t=0}Q(\Omega_{t})=|\Omega_{t}|\frac{T(B_{1})D'(\Omega_{t})\big|_{t=0}-D(B_{1})T'(\Omega_{t})\big|_{t=0}}{\pi T(B_{1})^{2}}=0
	\end{align*}
	and 
	\begin{align}
		\label{eq:d2Q}
		\frac{d^{2}}{dt^{2}}\Big|_{t=0}Q(\Omega_{t})=\frac{1}{\pi}\frac{|B_{1}|}{T(B_{1})^{2}}\Big(D''(\Omega_{t})\big|_{t=0}T(B_{1})-T''(\Omega_{t})\big|_{t=0}D(B_{1})\Big).
	\end{align}
	We note that $u_{B_{1}}(x)=\frac{1}{4}(1-|x|^{2})$, thus we get
	\begin{align*}
		T(B_{1})=\frac{\pi}{8}, \qquad\mbox{and}\qquad D(B_{1})=\frac{\pi}{4}.
	\end{align*}
	It is well known that
	\begin{align*}
		\frac{d}{dt}\Big|_{t=0}T(\Omega_{t})=0\quad\mbox{and}\quad \frac{d^{2}}{dt^{2}}\Big|_{t=0}T(\Omega_{t})=2\int_{\partial B_{1}}h^{2}d\sigma -2\int_{B_{1}}|\nabla h|^{2}dx,
	\end{align*}
	where $h$ is the harmonic function in $B_1$ with $h+\nabla u_{B_1}\cdot \eta=0$ on $\partial B_1$. From this and Proposition \ref{pro:DM}, we have
	\begin{align*}
		\frac{d^{2}}{dt^{2}}\Big|_{t=0}Q(\Omega_{t})&=\frac{|B_{1}|}{8T(B_{1})^{2}}\Big(D''(\Omega_{t})\big|_{t=0}-2T''(\Omega_{t})\big|_{t=0}\Big)\\
		&=\frac{16}{\pi}\Big(\int_{B_{1}}\det(D^{2}h)dx-2\int_{\partial B_{1}}h^{2}d\sigma+2\int_{B_{1}}|\nabla h|^{2}dx\Big).
	\end{align*}
\end{proof}
Since $\eta$ is volume preserving, by \eqref{volumepreserving} and the boundary condition of $h$, $h$ has mean zero over $\partial B_1$. Thus to conclude Theorem \ref{thm:main}, it remains to prove the following:
\begin{lemma}
	\label{lem:harmonic_bmz}
	Let $h$ be a harmonic function in the unit disk $B_1$ with $\int_{\partial B_{1}}h d\sigma=0$. Then
	\begin{align}
		\label{eq:estimate}
		\int_{B_{1}}\det(D^{2}h)dx+2\int_{B_{1}}|\nabla h|^{2}dx-2\int_{\partial B_{1}}h^{2}d\sigma\leq 0
	\end{align}
	and the equality holds if and only if $h$ is a linear map.
\end{lemma}

\begin{proof}
	We firstly transform the integral in $B_{1}$ into the boundary $\partial B_{1}$.
	\begin{align*}
		&\int_{B_{1}}\det(D^{2}h)dx+2\int_{B_{1}}|\nabla h|^{2}dx-2\int_{\partial B_{1}}h^{2}d\sigma\\
		=&\int_{B_{1}}\frac{1}{2}\Big((\Delta h)^{2}-|\nabla^{2}h|^{2}\Big)dx +2\int_{B_{1}}|\nabla h|^{2}dx-2\int_{\partial B_{1}}h^{2}d\sigma\\
		=&\int_{B_{1}}-\frac{1}{4}\Delta(|\nabla h|^{2})+2\int_{B_{1}}|\nabla h|^{2}dx-2\int_{\partial B_{1}}h^{2}d\sigma\\
		=&-\int_{\partial B_{1}}\Big(\frac{1}{4}\partial_{\nu}(|\nabla h|^{2})-2h\partial_{\nu}h+2h^{2}\Big)d\sigma.
	\end{align*}
	
	In polar coordinates, $h(r,\theta)$ is a harmonic function and $\int_{0}^{2\pi}h(1,\theta)d\theta=0$. Therefore, the Fourier series of $h(1,\theta)$ is 
	\begin{align*}
		h(1,\theta)=\sum_{n\geq 1}a_{n}\cos n\theta+\sum_{n\geq 1}b_{n}\sin n\theta.
	\end{align*}
	Then, $h$ coincide with its harmonic extension 
	\begin{align*}
		h(r,\theta)=\sum_{n\geq 1}r^{n}(a_{n}\cos n\theta+b_{n}\sin n\theta).
	\end{align*}
	
	A direct calculation yields
	\begin{align*}
		h_{r}(1,\theta)&=\sum_{n\geq 1}na_{n}\cos n\theta+\sum_{n\geq 1}nb_{n}\sin n\theta\\
		h_{\theta}(1,\theta)&=\sum_{n\geq 1}-na_{n}\sin n\theta+\sum_{n\geq 1}n b_{n}\cos n\theta\\
		h_{rr}(1,\theta)&=\sum_{n\geq 1}n(n-1)a_{n}\cos n\theta+\sum_{n\geq 1}n(n-1)b_{n}\sin n\theta\\
		h_{r\theta}(1,\theta)&=\sum_{n\geq1}-n^{2}a_{n}\sin n\theta+\sum_{n\geq 1}n^{2}b_{n}\cos n\theta.
	\end{align*}
	
	Since $\partial_{\nu}=\partial_{r}$ on $\partial B_{1}$, we have
	\begin{align*}
		&\int_{B_{1}}\det(D^{2}h)dx+2\int_{B_{1}}|\nabla h|^{2}dx-2\int_{\partial B_{1}}h^{2}d\sigma\\
		=&-\int_{\partial B_{1}}\Big(\frac{1}{4}\partial_{\nu}(|\nabla h|^{2})-2h\partial_{\nu}h+2h^{2}\Big)d\sigma\\
		=&\int_{0}^{2\pi}-\frac{1}{4}\partial_{r}\Big((h_{r})^{2}+r^{-2}(h_{\theta})^{2}\Big)+2h h_{r}-2h^{2}d\theta\\
		=&-\pi\sum_{n\geq 1}(n^{2}-2)(n-1)(a_{n}^{2}+b_{n}^{2})\leq 0.
	\end{align*}
	From the last equality above, we know that \eqref{eq:estimate} holds equality if and only if $$h(r,\theta)=ar\cos \theta+br\sin \theta,$$
	where $a, b\in \R$ are constants. In this case, $h$ is linear in $B_1$.
\end{proof}
	
Note that from the boundary condition of $h$ in Proposition \ref{pro:DQ}, once $h$ is linear, $F_t$ is a local translation map of $B_1$. Hence combining Proposition \ref{pro:DM}, Proposition \ref{pro:DQ} and Lemma \ref{lem:harmonic_bmz}, we immediately conclude Theorem \ref{thm:main}.

\vskip 0.2cm
Summarizing the extremal properties of $Q(\cdot)$ obtained in this section, as well as some numerical results, we make the following conjecture. 
\begin{conjecture}
	Among all bounded smooth convex planar domains, the supremum of $Q(\cdot)$ is attained only at disks.
\end{conjecture}

Last, we have also computed the derivative formula for $|\cdot|\lambda_1(\cdot)$ in two dimensions along the curve shortening flow, and similar monotonicity question motivates another conjecture of us:
\begin{conjecture}
	Among all bounded smooth convex planar domains, the supremum of the following functional
	\begin{align}
		E(\Omega):=\frac{|\Omega|\int_{\Omega} det(D^2 v)\, dx}{\pi \lambda_1(\Omega)}
	\end{align}
	is attained only at disks. In the above, $v$ is the first eigenfunction of the Dirichlet Laplacian on $\Omega$, with $\int_{\Omega} v^2 \, dx=1$.
\end{conjecture}

	\section{Monotonicity on rectangles via flow approach}

	We first prove Proposition \ref{not-equidistribution-thmintro}, which is crucial in the sign analysis in the proof of Theorem \ref{recderipos}. Even though by separation of variables, there is a series expression of the torsion function on rectangles, it is rather complicated to draw the conclusion this way. Instead, our proof will be based on the reflection argument and symmetry observation.
	
	\begin{proof}[Proof of Proposition \ref{not-equidistribution-thmintro}]
		Without loss of generality, we assume that $A=(0,0)$, $B=(a,0)$, $C=(a,b)$ and $D=(0,b)$, with $a>b$. Let $E$ be the point of the intersection of the line $l: \, y=x$ and the side $CD$. Let $M$ be the midpoint of $AD$, $N$ be the midpoint of $AB$ and $M'\in AN$ with $|AM'|=|AM|$. See Figure \ref{fig-rectangle-4} below. 
		
		\begin{figure}[htp]
			\centering
			\begin{tikzpicture}[scale = 2]
				\pgfmathsetmacro\Ls{2};
				\pgfmathsetmacro\hs{1};
				\pgfmathsetmacro\ts{0.3};
				\pgfmathsetmacro\hts{\hs*(1+\ts)};
				
				\fill (0, 0) circle (0.02 ) node[below] {\small $A$};
				\fill (0, 1) circle (0.02 ) node[left] {\small$D$};
				\fill (0, 0.3) circle (0.02 ) node[left] {\small $P$};
				\fill (0.3, 0) circle (0.02 ) node[below] {\small $P'$};
				\fill (1.5,0) circle (0.02) node[below] {\small$B$};
				\fill (1.5, 1) circle (0.02) node[right] {\small$C$};
				\fill (0,0.5) circle (0.02) node[left]{\small $M$};
				\fill (0.5,0) circle (0.02) node[below]{\small $M'$};
				\fill (1,1) circle (0.02) node[above]{\small $E$};
				\fill (0.75,0) circle (0.02) node[below]{\small $N$};

				\draw[dashed] (0,0)--(1.3,1.3);
				
				\fill (1.3,1.3) node[right]{\small $l$};
				
				\draw[->] (-0.2,0)--(1.8,0) node[right] {$x$};
				\draw[->](0,-0.2)--(0,1.2) node[right] {$y$};
				\draw (0,1)--(1.5,1);
				\draw  (1.5,0)--(1.5,1);
				
			\end{tikzpicture}
			\caption{Picture illustration of the proof of Proposition \ref{not-equidistribution-thmintro}}
			\label{fig-rectangle-4}
		\end{figure}
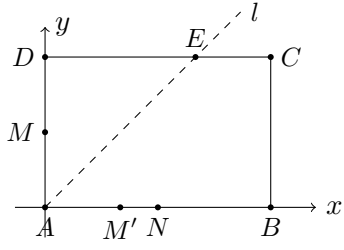
		
		Since the reflection of $DE$ about the line $l$ strictly lies inside the rectangle $\square ABCD$, for any $P\in AM$, $P'\in AM'$ with $|AP|=|AP'|$, similar to the proof of Theorem \ref{symmetrycomparison}, we have $|\nabla u(P)|<|\nabla u(P')|$, see also \cite[Theorem 5.1]{LY24} where a full proof of this fact is presented. Hence
		\begin{align*}
			\frac{1}{|AM|}\int_{AM}|\nabla u|^2\, ds<\frac{1}{|AM'|}\int_{AM'}|\nabla u|^2\, ds.
		\end{align*}
		Also, by reflection argument (similar to the argument in the proof of \cite[Proposition 3.2]{4L24}), $|\nabla u|$ is strictly increasing along the segment $AN$ from $A$ to $N$. Hence
		\begin{align*}
			\frac{1}{|AM'|}\int_{AM'}|\nabla u|^2\, ds<\frac{1}{|M'N|}\int_{M'N}|\nabla u|^2\, ds.
		\end{align*}
		Therefore,
		\begin{align*}
			\frac{1}{|AN|}\int_{AN}|\nabla u|^2\, ds=&\frac{1}{|AM'|+|M'N|}\left(\int_{AM'}|\nabla u|^2\, ds+\int_{M'N}|\nabla u|^2\, ds\right)\\
			>&\frac{1}{|AM|}\int_{AM}|\nabla u|^2\, ds.
		\end{align*}
		By symmetry, 
		\begin{align*}
			\frac{1}{|AN|}\int_{AN}|\nabla u|^2\, ds= \frac{1}{|AB|}\int_{AB}|\nabla u|^2\, ds,\quad  \frac{1}{|AM|}\int_{AM}|\nabla u|^2\, ds= \frac{1}{|AD|}\int_{AD}|\nabla u|^2\, ds.
		\end{align*}
		Hence \eqref{rectangle-neumann-data} is obtained.
	\end{proof}

	Now we are ready to prove Theorem \ref{recderipos}. The method is also motivated by the stretching flow argument in the proof of Theorem \ref{yangsheng2}.
	
	\begin{proof}[Proof of Theorem \ref{recderipos}]
		Let 
		\begin{align*}
			\eta(t,x,y)=\left(\frac{x}{1+t},0\right),\quad F_t(x,y)=\left((1+t)x,y\right).
		\end{align*}
		Then $F_t(R)=R_t$ and 
		\begin{align*}
			\partial_t F_t(x,y)=\eta(t,F(x,y)),\quad F_0(x,y)=(x,y).
		\end{align*}
		
		Without loss of generality, we assume that $b=1$. By \eqref{boundarytorsion}, Theorem \ref{Tshapederivative}, and the fact that $\eta\cdot \nu=0$ on $\partial R_t\setminus B_tC_t$ and $(x,y)\cdot \nu=0$ on $AB_t\cup AD$, we have
		\begin{align*}
			&\frac{d}{dt}\left(\frac{T(R_t)}{|R_t|^2}\right)=\frac{d}{dt}\left(\frac{T(R_t)}{(1+t)^2}\right)=\frac{\int_{B_tC_t}|\nabla u(t)|^2 \left(\frac{x}{1+t},0\right)\cdot (1,0) \, ds}{(1+t)^2}\\
			&-\frac{2}{(1+t)^3}\frac{1}{4}\left(\int_{B_tC_t}|\nabla u(t)(x,y)|^2(x,y)\cdot (1,0)\, ds+\int_{C_tD}|\nabla u(t)(x,y)|^2 (x,y)\cdot (0,1)\, ds\right)\\
			=&\frac{1}{2}\frac{1}{(1+t)^2}\left(\int_{B_tC_t}|\nabla u(t)|^2\, ds-\frac{1}{1+t}\int_{C_tD}|\nabla u(t)|^2\, ds\right),\\
			&\mbox{since on $B_tC_t, \, x=1+t$, and on $C_tD,\, y=1$,}\\
			=&\frac{1}{2}\frac{1}{(1+t)^2}\left(\frac{1}{|B_tC_t|}\int_{B_tC_t}|\nabla u(t)|^2\, ds-\frac{1}{|C_tD|}\int_{C_tD}|\nabla u(t)|^2\, ds\right).
		\end{align*}
		By Proposition \ref{not-equidistribution-thmintro}, $\tfrac{d}{dt}(T(R_t)/|R_t|^2)>0$ for $t>0$.   
	\end{proof}
	
	To conclude this section, motivated by Proposition \ref{not-equidistribution-thmintro}, we raise the following partially over-determined question, which remains open to us: 
	\begin{question}
		\label{harderrigidity}
		Let $Q$ be a quadrilateral with vertices $A_1, A_2,\cdots, A_4$. Let $A_{5}=A_1$ and $u$ be the torsion function over $Q$. If
		\begin{align*}
			\frac{1}{|A_iA_{i+1}|}\int_{A_iA_{i+1}}|\nabla u|^2\, ds\equiv \frac{1}{|\partial Q|}\int_{\partial Q}|\nabla u|^2\, ds, \quad \forall i=1,2,3,4,
		\end{align*}
	then is it true that $Q$ must be of a kite shape?
	\end{question}
	One can actually also consider similar problems for general polygons with $N$ sides, $N\ge 5$. We think that such problems are worth exploring.
	
	\section{An alternative proof of Proposition \ref{rectangle-thm} via continuous symmetrization}
	We first recall the well-known Steiner symmetrization (see for example \cite[Chapter 2]{Henrot}), which is a powerful tool to solve geometric optimization problems.
	
	\begin{definition}[Steiner symmetrization]
		\label{def:Steiner symmetrization}
		Let $\Omega\subset\R^{2}$ be a measurable set, fix a line $l$ in $\R^{2}$ and denote $\tau$ the unit normal to $l$. Let $\Omega_{x}:=\{x+t\tau\in \Omega:t\in \R\}$ be the cross-section of $\Omega$ passing through $x$. Let $\Omega^{*}_{x}:=\{x+s\tau:|s|<\mathcal{H}^{1}(\Omega_{x})/2\}$ be the line segment centered at $x$ and perpendicular to $l$. We call the domain $\Omega^{*}:=\cup_{x\in l}\Omega^{*}_{x}$ the Steiner symmetrization of $\Omega$ with respect to the line $l$.
	\end{definition}
	
	Using the notation from Definition \ref{def:Steiner symmetrization}, the continuous Steiner symmetrization of a convex planar domain $\Omega$ with respect to $l$ can be defined by  $$\Omega^{t}:=\cup_{x\in l}\{z^{t}+s\nu: z^{t}=tx+(1-t)M(x), |s|<\mathcal{H}^{1}(\Omega_{x})/2\}.$$ Here $M(x)$ is the midpoint of $\Omega_{x}$. Intuitively, the continuous Steiner symmetrization $\Omega^{t}$ is a domain formed by transforming each line segment $\Omega(x)$ in the direction $\tau$ until its midpoint $M(x)$ coincides with the point $z^{t}$. Figure \ref{fig-rectangle-0} illustrates the action of continuous Steiner symmetrization on triangles.
	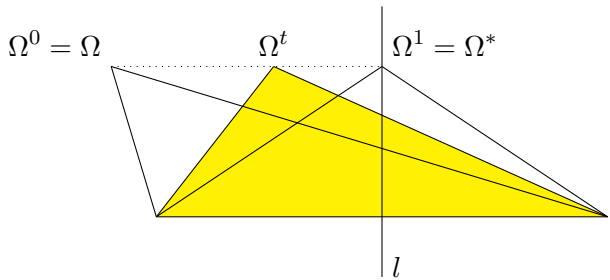
\begin{figure}[htp]
		\centering
		\begin{tikzpicture}[scale = 2]
			\pgfmathsetmacro\L{1.5};
			\pgfmathsetmacro\h{1};
			\pgfmathsetmacro\t{0.4};
			\pgfmathsetmacro\x{-\L*1.2};
			\pgfmathsetmacro\xt{\t*\x};
			
			\fill[yellow,opacity=0.5] (-\L,0)--(\L,0)--(\xt,\h)--cycle;
			\draw (-\L,0)--(\L,0);
			\draw (\L,0)--(\xt,\h);
			\draw (\xt,\h)--(-\L,0);
			\node[above] at (\xt, \h) {$\Omega^{t}$};
			
			\draw (0,-0.4*\h)--(0,1.4*\h);
			\node[below right] at (0, -0.2*\h) {$l$};
			\draw[dotted] (\x,\h)--(0,\h);
			
			\draw (-\L,0)--(\x,\h);
			\draw (\x,\h)--(\L,0);
			\node[above left] at (\x, \h) {$\Omega^{0}=\Omega$};

			\draw (-\L,0)--(0,\h);
			\draw (0,\h)--(\L,0);
			\node[above right] at (0, \h) {$\Omega^{1}=\Omega^{*}$};
		\end{tikzpicture}
		\caption{Continuous Steiner symmetrization}
		\label{fig-rectangle-0}
	\end{figure}

	The continuous Steiner symmetrization, which is a generalization of the Steiner symmetrization, was first introduced by P\'{o}lya and Szeg\"{o}\cite{PS51}, dealing with the monotonicity of shape functional under continuous deformations. Different versions of it have been studied by Solynin\cite{Solynin92} and Brock\cite{Brock95,Brock00}. The difference in cases of convex domains is only in the way of defining "continuity parameters". The version studied in \cite{Brock95,Brock00} is more general and it works even for non-connected domains. We refer the readers to these papers and to the excellent book \cite{Henrot} and the recent paper \cite{Solynin20} for the properties of this transformation.

	An important property of the continuous Steiner symmetrization is that the torsional rigidity of $\Omega^{t}$ increases with respect to $t$ (see \cite{Brock95} or the classic monograph \cite{PS51}). With this fact, we present another proof of Proposition \ref{rectangle-thm}.

	\begin{proof}[Proof of Proposition \ref{rectangle-thm}]
		Let $\varphi(s)=T(R(s))$.
		For $s\in (0,1)$, we let the vertices of $R(s)$ be denoted by $A,B,C,D$, where $|AB|=s$ and $|AD|=s^{-1}$.

		In the first step, we apply continuous Steiner symmetrization(CSS) to $R(s)$ with respect to the line $l_{1}$, which is the perpendicular bisector of the diagonal $BD$. This process generates a family of parallelograms $\{R(s)^{t}\}_{t\in[0,1]}$ with vertices $\{A_{t}, B_{t}, C_{t},D_{t}\}$, where $B_{t}=B$ and $D_{t}=D$. See Figure \ref{fig-rectangle-1}. 
		
		\begin{figure}[htp]
			\centering
			\begin{tikzpicture}[scale = 2]
				\pgfmathsetmacro\xlength{1.5};
				\pgfmathsetmacro\ylength{2};
				\pgfmathsetmacro\Mx{1/2*\xlength};
				\pgfmathsetmacro\My{1/2*\ylength};
				\pgfmathsetmacro\t{0.5};

				\fill (0, 0) circle (0.02) node[below left] {\small ${A}$};
				\fill (\xlength, 0) circle (0.02) node[below right] {\small ${B}$};
				\fill (\xlength, \ylength) circle (0.02) node[above right] {\small ${C}$};
				\fill (0, \ylength) circle (0.02) node[above left] {\small ${D}$};
				
				\draw[thick] (0,0)--(\xlength,0);
				\draw[thick] (\xlength,0)--(\xlength,\ylength);
				\draw[thick] (\xlength,\ylength)--(0,\ylength);
				\draw[thick] (0,0)--(0,\ylength);
				
				\draw[dashed]  (0,\ylength)-- (\xlength,0);
				\draw[dashed]  (\xlength,\ylength)-- (0,0);
				\draw[dashed] (\Mx+0.8*\ylength, \My+0.8*\xlength)--(\Mx-0.8*\ylength, \My-0.8*\xlength);
				\fill (\Mx+0.8*\ylength, \My+0.8*\xlength) circle (0.02) node[below right] {\small ${l_{1}}$};
				
				
				\pgfmathsetmacro\diagonal{(\xlength^(2)+\ylength^(2))^(1/2)};
				\pgfmathsetmacro\rAx{-\t*\xlength/\diagonal*(1/2*\diagonal-\xlength^(2)/\diagonal)};
				\pgfmathsetmacro\rAy{\t*\ylength/\diagonal*(1/2*\diagonal-\xlength^(2)/\diagonal)};
				\pgfmathsetmacro\rCx{\xlength+\t*\xlength/\diagonal*(1/2*\diagonal-\xlength^(2)/\diagonal)};
				\pgfmathsetmacro\rCy{\ylength-\t*\ylength/\diagonal*(1/2*\diagonal-\xlength^(2)/\diagonal)};
				
				\draw[->,thick] (\Mx+0.8*\xlength, \My)--(\Mx+1.6*\xlength,\My);
				\pgfmathsetmacro\trans{\Mx+1.8*\xlength};
				\node[below] at (\Mx+1.2*\xlength, \My*0.8) {\small CSS};

				\draw[dashed] (0+\trans,0)--(\xlength+\trans,0);
				\draw[dashed] (\xlength+\trans,0)--(\xlength+\trans,\ylength);
				\draw[dashed] (\xlength+\trans,\ylength)--(0+\trans,\ylength);
				\draw[dashed] (0+\trans,0)--(0+\trans,\ylength);
				
				\draw[dashed] (\Mx+0.8*\ylength+\trans, \My+0.8*\xlength)--(\Mx-0.8*\ylength+\trans, \My-0.8*\xlength);
				\fill (\Mx+0.8*\ylength+\trans, \My+0.8*\xlength) circle (0.02) node[below right] {\small ${l_{1}}$};

				\pgfmathsetmacro\diagonal{(\xlength^(2)+\ylength^(2))^(1/2)};
				\pgfmathsetmacro\rAx{-\t*\xlength/\diagonal*(1/2*\diagonal-\xlength^(2)/\diagonal)};
				\pgfmathsetmacro\rAy{\t*\ylength/\diagonal*(1/2*\diagonal-\xlength^(2)/\diagonal)};
				\pgfmathsetmacro\rCx{\xlength+\t*\xlength/\diagonal*(1/2*\diagonal-\xlength^(2)/\diagonal)};
				\pgfmathsetmacro\rCy{\ylength-\t*\ylength/\diagonal*(1/2*\diagonal-\xlength^(2)/\diagonal)};
				
				\fill (\rAx+\trans, \rAy) circle (0.02) node[below left] {\small ${A_{t}}$};
				\fill (\xlength+\trans, 0) circle (0.02) node[below right] {\small ${B_{t}}$};
				\fill (\rCx+\trans, \rCy) circle (0.02) node[above right] {\small ${C_{t}}$};
				\fill (0+\trans, \ylength) circle (0.02) node[above left] {\small ${D_{t}}$};
				
				\draw[thick] (\rAx+\trans, \rAy)--(\xlength+\trans, 0);
				\draw[thick] (\xlength+\trans, 0)--(\rCx+\trans, \rCy);
				\draw[thick] (\rCx+\trans, \rCy)--(0+\trans, \ylength);
				\draw[thick] (0+\trans, \ylength)--(\rAx+\trans, \rAy);
				
				\draw[dashed] (0+\trans, \ylength)--(\xlength+\trans, 0);
				\draw[dashed] (\rAx+\trans, \rAy)--(0+\trans,0);
				\draw[dashed] (\rCx+\trans, \rCy)--(\xlength+\trans,\ylength);

			\end{tikzpicture}
			\caption{Step 1}
			\label{fig-rectangle-1}
		\end{figure}
		
		In the second step, we apply Steiner symmetrization (SS) to $R(s)^{t}$ with respect to the line $l_{2}$, which is perpendicular to the side $A_{t}B_{t}$. The resulting shape is a rectangle $\hat{R}$ with vertices $\hat{A},\hat{B},\hat{C},\hat{D}$. See Figure \ref{fig-rectangle-2}. 
		\begin{figure}[htp]
			\centering
			\begin{tikzpicture}[scale = 2]
				\pgfmathsetmacro\xlength{1.5};
				\pgfmathsetmacro\ylength{2};
				\pgfmathsetmacro\Mx{1/2*\xlength};
				\pgfmathsetmacro\My{1/2*\ylength};
				\pgfmathsetmacro\t{0.5};
				
				\pgfmathsetmacro\diagonal{(\xlength^(2)+\ylength^(2))^(1/2)};
				\pgfmathsetmacro\rAx{-\t*\xlength/\diagonal*(1/2*\diagonal-\xlength^(2)/\diagonal)};
				\pgfmathsetmacro\rAy{\t*\ylength/\diagonal*(1/2*\diagonal-\xlength^(2)/\diagonal)};
				\pgfmathsetmacro\rCx{\xlength+\t*\xlength/\diagonal*(1/2*\diagonal-\xlength^(2)/\diagonal)};
				\pgfmathsetmacro\rCy{\ylength-\t*\ylength/\diagonal*(1/2*\diagonal-\xlength^(2)/\diagonal)};
				
				\pgfmathsetmacro\diagonal{(\xlength^(2)+\ylength^(2))^(1/2)};
				\pgfmathsetmacro\rAx{-\t*\xlength/\diagonal*(1/2*\diagonal-\xlength^(2)/\diagonal)};
				\pgfmathsetmacro\rAy{\t*\ylength/\diagonal*(1/2*\diagonal-\xlength^(2)/\diagonal)};
				\pgfmathsetmacro\rCx{\xlength+\t*\xlength/\diagonal*(1/2*\diagonal-\xlength^(2)/\diagonal)};
				\pgfmathsetmacro\rCy{\ylength-\t*\ylength/\diagonal*(1/2*\diagonal-\xlength^(2)/\diagonal)};
				
				\fill (\rAx, \rAy) circle (0.02) node[below left] {\small ${A_{t}}$};
				\fill (\xlength, 0) circle (0.02) node[below right] {\small ${B_{t}}$};
				\fill (\rCx, \rCy) circle (0.02) node[above right] {\small ${C_{t}}$};
				\fill (0, \ylength) circle (0.02) node[above left] {\small ${D_{t}}$};
				
				\draw[thick] (\rAx, \rAy)--(\xlength, 0);
				\draw[thick] (\xlength, 0)--(\rCx, \rCy);
				\draw[thick] (\rCx, \rCy)--(0, \ylength);
				\draw[thick] (0, \ylength)--(\rAx, \rAy);
				
				\draw[->,thick] (\Mx+0.8*\xlength, \My)--(\Mx+1.6*\xlength,\My);
				\pgfmathsetmacro\trans{\Mx+1.8*\xlength};
				\node[below] at (\Mx+1.2*\xlength, \My*0.8) {\small SS};
				
				\draw[dashed] (\rAx+\trans,\rAy)--(\xlength+\trans,0);
				\draw[dashed] (\xlength+\trans, 0)--(\rCx+\trans, \rCy);
				\draw[dashed] (\rCx+\trans, \rCy)--(0+\trans, \ylength);
				\draw[dashed] (0+\trans, \ylength)--(\rAx+\trans, \rAy);
				
				\pgfmathsetmacro\AtBt{((\rAx-\xlength)^(2)+(\rAy)^(2))^(1/2)};
				\pgfmathsetmacro\taux{\rAy/\AtBt};
				\pgfmathsetmacro\tauy{(\xlength-\rAx)/\AtBt};
				\pgfmathsetmacro\width{\xlength*\ylength/\AtBt};
				
				\pgfmathsetmacro\Ahx{\rAx+\trans};
				\pgfmathsetmacro\Ahy{\rAy};
				\pgfmathsetmacro\Bhx{\xlength+\trans};
				\pgfmathsetmacro\Bhy{0};
				\pgfmathsetmacro\Chx{\Bhx+\width*\taux};
				\pgfmathsetmacro\Chy{\Bhy+\width*\tauy};
				\pgfmathsetmacro\Dhx{\rAx+\width*\taux+\trans};
				\pgfmathsetmacro\Dhy{\rAy+\width*\tauy};

				\fill (\Ahx, \Ahy) circle (0.02) node[below left] {\small ${\hat{A}}$};
				\fill (\Bhx,\Bhy) circle (0.02) node[below right] {\small ${\hat{B}}$};
				\fill (\Chx, \Chy) circle (0.02) node[above right] {\small ${\hat{C}}$};
				\fill (\Dhx, \Dhy) circle (0.02) node[above left] {\small ${\hat{D}}$};

				\draw[thick] (\Ahx, \Ahy)--(\Bhx,\Bhy);
				\draw[thick] (\Bhx,\Bhy)--(\Chx,\Chy);
				\draw[thick] (\Dhx,\Dhy)--(\Chx,\Chy);
				\draw[thick] (\Dhx,\Dhy)--(\Ahx,\Ahy);
				
				\draw[dashed] (\rAx+1.3*\width*\taux, \rAy+1.3*\width*\tauy)--(\rAx-0.3*\width*\taux,\rAy-0.3*\width*\tauy);
				\fill (\rAx+1.3*\width*\taux, \rAy+1.3*\width*\tauy) circle (0.02) node[above left] {\small ${l_{2}}$};
				\draw[dashed] (\rAx+1.3*\width*\taux+\trans, \rAy+1.3*\width*\tauy)--(\rAx-0.3*\width*\taux+\trans,\rAy-0.3*\width*\tauy);
				\fill (\rAx+1.3*\width*\taux+\trans, \rAy+1.3*\width*\tauy) circle (0.02) node[above left] {\small ${l_{2}}$};
				
			\end{tikzpicture}
			\caption{Step 2}
			\label{fig-rectangle-2}
		\end{figure}
		
		From the above process, after applying a rotation and transformation to $\hat{R}$, we obtain a rectangle $R(\xi(s;t))$, where $\xi(s;t)=|A_tB_t|$. In fact, using elementary geometry argument, we have 
		\begin{itemize}
			\item[(1)] $|CC_{t}|=t\cdot \lambda(s)$, where $\lambda(s)=\tfrac{1}{2}\tfrac{1-s^4}{1+s^4}(s^{-2}+s^2)^{1/2}$ is the distance from $C$ to the line $l_1$.
			\item[(2)] $|A_{t}B_{t}|^{2}=t^{2}\lambda(s)^{2}+s^{2}+2t\lambda(s)s^{2}(s^{2}+s^{-2})^{-1/2}$.
			\item[(3)] $\xi(s;t)=|A_{t}B_{t}|$.
		\end{itemize}
		At each step, by properties of (CSS) and (SS), the torsional rigidity is strictly increasing. Therefore, we have 
		\begin{align*}
			\varphi(\xi(s;t))>\varphi(\xi(s;0)),\quad \forall t\in[0,1].
		\end{align*}
		Note that $\xi(s;t)$ is a strictly increasing smooth function with respect to $t$, with $\xi(s;0)=s$ and $s<\xi(s;1)<1$. Hence $\varphi(s)$ is strictly increasing in the interval $(\xi(s;0),\xi(s;1))$. Since $s\in (0,1)$ is arbitrary, $\varphi(s)$ is strictly increasing with respect to $s\in(0,1)$.
	\end{proof}
	
	\section{Further remarks on Question \ref{kq1}}
	
	First, motivated by Question \ref{kq1}, we observe the monotonicity property of the isoperimetric ratio along the inverse mean curvature flow (IMCF), and therefore the classical Euclidean isoperimetric can also be proved by the inverse mean curvature flow. 
	
	Recall that the inverse mean curvature flow is a one-parameter family of closed, embedded hypersurfaces \( M_t \subset \mathbb{R}^{n+1} \) (or in a Riemannian manifold) 
	satisfying the evolution equation
	\[
	\frac{\partial}{\partial t} x = \frac{1}{H} \, \nu, \quad t \in [0, \infty),
	\]
	where \( x \) is the position vector of \( M_t \), \( H > 0 \) is the mean curvature, and \( \nu \) is the outward unit normal vector field. 
	This flow expands the hypersurface in the direction of its normal with speed inversely 
	proportional to its mean curvature.
	
	\begin{proposition}
		\label{IMCFproof}
		Let $\Omega$ be a smooth strictly star-shaped mean convex domain, and let $\Omega_t$ be the domain along the inverse mean curvature flow at a time $t$. Then,
		\begin{align*}
			\frac{d}{dt}\left(\frac{P^{\frac{n}{n-1}}(\Omega_t)}{|\Omega_t|}\right)\le 0.
		\end{align*}
	\end{proposition}
	
	\begin{proof}
		By \cite{Ger90}, the strict star-shaped mean convex property is preserved along IMCF. By direct computation, we have
		\begin{align*}
			\frac{d}{dt}\frac{P^{\frac{n}{n-1}}(\Omega_t)}{|\Omega_t|}=& \frac{1}{|\Omega_t|^2} \left(\left(\frac{n}{n-1}P^{\frac{1}{n-1}}(\Omega_t) \int_{\partial \Omega_t}H\cdot\frac{1}{H}\, d\sigma_t\right)|\Omega_t|-P^{\frac{n}{n-1}}(\Omega_t) \int_{\partial \Omega_t}\frac{1}{H}\, d\sigma_t\right) \\
			=&\frac{P^{\frac{n}{n-1}}(\Omega_t)}{|\Omega_t|^2} \cdot \left(\frac{n}{n-1}|\Omega_t|-\int_{\partial\Omega_t}\frac{1}{H}\, d\sigma\right)\le 0,
		\end{align*}
		where the last inequality is due to Ros, see \cite{Ros}.
	\end{proof}
	The conclusion is still valid if the initial surface is smooth outward minimizing (i.e., variational mean convex), due to Husiken and Ilmanen \cite{HI01}, where the terminology outward minimizing is also explained there.

	For the torsional rigidity, Question \ref{kq1} motivates us to consider the following flow given by 
	\begin{align}
		\label{nonlocal}
		\frac{\partial X}{\partial t}(t,x)=-\frac{1}{|\nabla u(t)(X(t,x))|^2} (X(t,x)\cdot \nu)\nu,\quad X(0,x)=x,
	\end{align}
	where $\nu$ is the unit outer normal, $X(t,x)$ is the position vector and $u(t)$ is the torsion function over the domain enclosed by $X(t,\partial \Omega)$, denoted by $\Omega_t$。
	
	\begin{proposition}
		\label{newzhenxi}
		Let $\Omega$ be a smooth convex domain, $X(t,x)$ be the flow given by \eqref{nonlocal} and suppose that $\Omega_t$ is convex for any $t\in [0,T)$, then $T(\Omega_t)/|\Omega_t|^{(n+2)/n}$ is monotonically increasing for $t\in [0,T)$.	\end{proposition}
	
	\begin{proof}
		Let $\zeta$ be the normal speed of $X(t,\cdot)$. That is, $\zeta \in C^\infty(\partial \Omega_t)$ with $\zeta(z)=X(t,z)\cdot \nu(z)$, for any $z\in \partial \Omega_t$. Then by direct computation,
		\begin{align*}
			\frac{d}{dt}\left(\frac{T(\Omega_t)}{|\Omega_t|^{(n+2)/n}}\right)\ge 0
		\end{align*} is equivalent to 
		\begin{align*}
			\left(\int_{\partial \Omega_t}|\nabla u(t)|^2 \zeta \, d\sigma\right)|\Omega_t| \ge \frac{n+2}{n}T(\Omega_t)\int_{\partial \Omega_t}\zeta \, d\sigma.
		\end{align*}
		By \eqref{boundarytorsion}, the above inequality is equivalent to 
		\begin{align*}
			\left(\int_{\partial \Omega_t}|\nabla u(t)|^2 \zeta \, d\sigma\right)|\Omega_t| \ge\frac{1}{n}\left(\int_{\partial \Omega_t}|\nabla u(t)|^2\, (x\cdot \nu)\, d\sigma\right)\int_{\partial \Omega_t}\zeta \, d\sigma.
		\end{align*}
		Plugging 
		\begin{align*}
			\zeta=-\frac{1}{|\nabla u(t)|^2} (x\cdot \nu),
		\end{align*}and since 
		\begin{align*}
			|\Omega_t|=\frac{1}{n}\int_{\partial \Omega_t}(x\cdot \nu)\, d\sigma,
		\end{align*}
		the above is equivalent to
		\begin{align*}
			\left(\int_{\partial \Omega_t} (x\cdot \nu)\, d\sigma\right)^2 \le \left(\int_{\partial \Omega_t} |\nabla u(t)|^2(x\cdot \nu)\, d\sigma\right)\left(\int_{\partial \Omega_t} \frac{1}{|\nabla u(t)|^2}(x\cdot \nu)\, d\sigma\right).
		\end{align*}
		This is indeed valid, applying the Cauchy-Schwarz inequality with respect to the measure $d\mu=(x\cdot\nu)d\sigma$. Note that for a smooth convex hypersurface, $x\cdot\nu \ge0$, and thus $d\mu$ is a measure. 
		
	\end{proof}
	
	We have the following conjecture: 
	\begin{conjecture}
		\label{nonlocalflow}
		Let $\Omega$ be a bounded smooth convex domain in $\mathbb{R}^n$. Then, $X(t,\partial\Omega)$ converges to a round point in $C^\infty$ norm at a finite time $T$, and during the evolution, the convexity property is preserved.
	\end{conjecture}
	If the conjecture is validated, then in view of Proposition \ref{newzhenxi}, Question \ref{kq1} is answered under the configuration of the flow given by \eqref{nonlocal} within smooth convex frameworks.
	
	The construction of \eqref{nonlocal} is also motivated by the work of Guan-Li \cite{GL15}, where they prove the monotonicity of $P(\cdot)$ under the normalized mean curvature type flow preserving volume, given by 
	\begin{align*}
		\label{nonlocal}
		\frac{\partial X}{\partial t}=\left((n-1)-H (X\cdot \nu)\right)\nu.
	\end{align*}
	However, the flow \eqref{nonlocal} seems to be more difficult to handle, since $|\nabla u|$ is nonlocal over the boundary. Therefore, even disregarding Question \ref{kq1}, resolving Conjecture \ref{nonlocalflow} itself deserves further research.

	\vskip 0.3cm
	
	\noindent
	\textbf{Data Availibility Statement:} This study is a theoretical analysis, and no new data were created or analyzed. Therefore, data sharing is not applicable to this article.


\end{document}